\documentclass[12pt]{amsart}

\usepackage{amssymb}
\usepackage{bbm}
\usepackage{amsfonts}
\usepackage{latexsym}
\usepackage[all]{xy}
\usepackage{amscd}
\usepackage{comment}
\usepackage[headings]{fullpage}
\usepackage{pstricks}
\usepackage[bookmarks=false, colorlinks, linkcolor=mylinkcolor, citecolor=mycitecolor]{hyperref}

\definecolor{mycitecolor}{rgb}{0,.6,0} 
\definecolor{mylinkcolor}{rgb}{0,0,.8}   

\numberwithin{equation}{section}
\newtheorem{theorem}[equation]{Theorem}
\newtheorem{lemma}[equation]{Lemma}
\newtheorem{proposition}[equation]{Proposition}
\newtheorem{corollary}[equation]{Corollary}
\theoremstyle{definition}

\newcommand{\DOT}{\setlength{\unitlength}{1pt}\begin{picture}(2.5,2)
                  (1,1)\put(2,3.5){\circle*{2}}\end{picture}}

\newcommand{\bu}{\DOT}
\newcommand{\la}{\langle}
\newcommand{\ra}{\rangle}
\newcommand{\ot}{\otimes}
\newcommand{\Wedge}{\textstyle\bigwedge}
\newcommand{\lexp}[2]{{\vphantom{#2}}^{#1}{#2}}

\newcommand{\N}{\mathbb{N}}

\newcommand{\Z}{\mathbb{Z}}
\renewcommand{\O}{\{0,1\}}
\newcommand{\HHD}{\HH^{\bu}}

\newcommand{\CC}{\mathbb{C}}
\newcommand{\R}{\mathcal{R}}
\newcommand{\ds}{\displaystyle}
\newcommand{\mH}{\mathcal{H}}

\newcommand{\CX}{\CC \langle X \rangle}
\newcommand{\X}{\langle X \rangle}
\newcommand{\irr}{\text{irr}}
\newcommand{\q}{\mathbf{q}}
\newcommand{\rt}{\triangleright}

\DeclareMathOperator{\ddet}{det}

\DeclareMathOperator{\Hom}{Hom}
 
\DeclareMathOperator{\HH}{HH}
\DeclareMathOperator{\Ext}{Ext}

\DeclareMathOperator{\gr}{gr}
\DeclareMathOperator{\id}{id}
\DeclareMathOperator{\mult}{mult}

\renewcommand{\span}{\operatorname{span}}
\renewcommand{\H}{\operatorname{H}}

\begin{document}

\title[Twisted quantum Drinfeld Hecke algebras]
{Twisted quantum Drinfeld Hecke algebras}

\date{Wednesday, August 08, 2012}

\author{Deepak Naidu}
\email{dnaidu@math.niu.edu}
\address{Department of Mathematical Sciences, Northern Illinois 
University, DeKalb, Illinois 60115, USA}

\begin{abstract}
We generalize quantum Drinfeld Hecke algebras by incorporating a 
$2$-cocycle on the associated finite group.
We identify these algebras as specializations of deformations of 
twisted skew group algebras, giving an explicit connection to Hochschild cohomology.
We classify these algebras for diagonal actions, as well 
as for the symmetric groups with their natural representations. 
Our results show that the parameter spaces for the symmetric groups
in the twisted setting is smaller than in the untwisted setting.
\end{abstract}

\maketitle

\begin{section}{Introduction}

Drinfeld Hecke algebras were defined by V.\ Drinfeld in the paper \cite{D}.
They arise as symplectic reflection algebras in the work 
of P.\ Etingof and V.\ Ginzburg \cite{EG}, as braided Cherednik algebras in the work of
Y.\ Bazlov and A.\ Berenstein \cite{BB}, and as graded version of affine Hecke algebras
in the work of G.\ Lusztig \cite{L}. They arise in diverse areas, such
as representation theory, combinatorics, and orbifold theory, and they were used by
I.\ Gordon to prove a version of the $n!$ conjecture for Weyl groups \cite{G}.

In this paper, we consider quantum and twisted analogs of Drinfeld Hecke algebras
by incorporating quantum parameters as well as a $2$-cocycle on the associated finite group.
We simultaneously generalize twisted Drinfeld Hecke algebras and quantum Drinfeld Hecke algebras.
The former was studied by S.\ Witherspoon in \cite{W}, and the latter was studied by 
V.\ Levandovskyy and A.\ Shepler in \cite{L}, and by S.\ Witherspoon and the author in \cite{NW}. 
In \cite{C}, T.\ Chmutova generalized symplectic reflection algebras by incorporating
a $2$-cocycle on the associated finite group, and showed that such a $2$-cocycle
arises naturally for nonfaithful representations. Such a $2$-cocycle also arises
in orbifold theory, where they are known as discrete torsion \cite{AR, CGW, V}.

Let $V$ be a complex vector space with basis $v_1,v_2,\ldots,v_n$,
and let $\q:= (q_{ij})_{1 \leq i,j \leq n}$ be a tuple of
nonzero scalars for which $q_{ii}=1$ and $q_{ji}=q_{ij}^{-1}$ for all $i,j$.
Let $S_\q(V)$ denote the {\bf quantum symmetric algebra}:
\[
 S_\q(V) := \CC\la v_1,\ldots,v_n \mid v_iv_j = q_{ij}v_jv_i  \mbox{ for all }
      1\leq i,j\leq n \ra.
\]

Let $G$ be a finite group acting linearly on $V$, and let $\alpha:G \times G \to \CC^\times$ 
be a normalized $2$-cocycle on $G$.
Let $\kappa: V\times V\rightarrow \CC^\alpha G$ be a bilinear map for which 
$\kappa(v_i,v_j) = - q_{ij}\kappa(v_j,v_i)$ for all $1\leq i,j\leq n$. 
Let $T(V)$ be the tensor algebra on $V$, and define  
\[
  \mH_{\q, \kappa, \alpha} := T(V)\#_\alpha G/(  v_iv_j - q_{ij} v_jv_i - \kappa(v_i,v_j) \mid
     1\leq i,j\leq n),
\]
the quotient of the twisted skew group algebra 
$T(V)\#_\alpha G$ by the ideal generated by all elements of the form specified.
Suppose that the action of $G$ on $V$ induces an action of $G$ on
$S_\q(V)$ by automorphisms, so we may form the twisted skew group algebra $S_\q(V)\#_\alpha G$.
Assigning each $v_i$ degree one and each group element degree zero makes $\mH_{\q, \kappa, \alpha}$ a
filtered algebra, and makes $S_\q(V)\#_\alpha G$ a graded algebra.
We will call $\mH_{\q, \kappa, \alpha }$ a {\bf twisted quantum Drinfeld Hecke algebra} (over $\CC$) if
it satisfies the Poincar\'e-Birkhoff-Witt condition:
The associated graded algebra $\gr \mH_{\q, \kappa, \alpha}$ 
is isomorphic, as a graded algebra, to $S_\q(V)\#_\alpha G$.
The space of all maps $\kappa: V\times V\rightarrow \CC^\alpha G$ for which
$\mH_{\q, \kappa, \alpha }$ is a twisted quantum Drinfeld Hecke algebra
will be referred to as the {\bf parameter space}.

\vspace{0.1in}

\noindent {\bf Main results and organization:} 

\vspace{0.1in}

In Section~\ref{necc and suff conditions}, we use G.\ Bergman's Diamond Lemma \cite{B} 
to give necessary and sufficient conditions for the algebra $\mH_{\q, \kappa, \alpha }$
to be a twisted quantum Drinfeld Hecke algebra.

In Section~\ref{deformations}, we identify the twisted quantum Drinfeld Hecke algebras 
$\mH_{\q, \kappa, \alpha}$ as specializations of particular types of deformations 
of the twisted skew group algebras $S_\q(V)\#_\alpha G$.  

Section~\ref{computing HH^2} develops the homological algebra needed for 
the sections that follow. Specifically, this section is concerned with the 
computation of the degree two Hochschild cohomology of $S_\q(V)\#_\alpha G$.

In Section~\ref{constant cocycles}, we establish a one-to-one correspondence between the
subspace of {\em constant} Hochschild $2\text{-cocycles}$ (defined in Section~\ref{deformations}) 
contained in $\HH^2(S_\q(V)\#_\alpha G)$ and 
twisted quantum Drinfeld Hecke algebras associated to the quadruple $(G,V, \q, \alpha)$. 
As as consequence, we show that every constant Hochschild $2$-cocycles on $S_\q(V)\#_\alpha G$
lifts to a deformation of $S_\q(V)\#_\alpha G$.

In Section~\ref{diagonal}, we consider diagonal actions of $G$ on a chosen basis for $V$, and
using results from \cite{NSW} we classify the corresponding twisted quantum Drinfeld Hecke algebras.

In Section~\ref{natural}, we consider the symmetric groups $S_n, \, n \geq 5$, 
with their natural representations, with the unique nontrivial quantum parameters $q_{ij}=-1$, $i \neq j$,
and with a cohomologically nontrivial $2$-cocycle on $S_n$, which is unique up to coboundary. 
We classify the corresponding twisted quantum Drinfeld Hecke algebras.
Our results show that the parameter space in the twisted setting is smaller than in the untwisted setting.

Throughout the paper, let $G$ denote a finite group
acting linearly on a complex vector space $V$ with basis $v_1,v_2,\ldots,v_n$.
Let $\q:= (q_{ij})_{1 \leq i,j \leq n}$ denote a tuple of
nonzero scalars for which $q_{ii}=1$ and $q_{ji}=q_{ij}^{-1}$ for all $i,j$.
We will work over the complex numbers $\CC$, and all tensor products will be
taken over $\CC$ unless otherwise indicated.

\end{section}
\begin{section}{Necessary and sufficient conditions} \label{necc and suff conditions}

In this section, we will use G.\ Bergman's Diamond Lemma \cite{B} to give necessary 
and sufficient conditions for the algebra $\mH_{\q, \kappa, \alpha }$
(defined in the introduction and recalled below) to be a twisted quantum Drinfeld Hecke algebra. First, 
we recall the notion of a twisted skew group algebra: Let $G$ be a finite group, and 
let $\alpha:G \times G \to \CC^\times$ 
be a normalized $2$-cocycle on $G$, that is
\[
\alpha(g_1,g_2)\alpha(g_1g_2,g_3) = \alpha(g_2,g_3)\alpha(g_1,g_2g_3)
\qquad \text{ and } \qquad \alpha(g,1) = 1 = \alpha(1,g), 
\]
for all $g,g_1,g_2,g_3 \in G$. 
Let $A$ be an algebra on which $G$ acts by automorphisms.
The {\bf twisted skew group algebra} $A \#_\alpha G$ is defined as follows.
As a vector space, $A \#_\alpha G$ is $A \ot \CC G$. 
Multiplication on $A \#_\alpha G$ is determined by 
\[
   (a \ot g) (b  \ot h) := \alpha(g,h) a (\lexp{g}{b}) \ot gh
\]
for all $a,b\in A$ and all $g,h\in G$, where a left superscript denotes
the action of the group element. 
The $2$-cocycle condition on $\alpha$ ensures that $A \#_\alpha G$ is an associative algebra. 
Note that $A$ is a subalgebra of $A \#_\alpha G$ via the isomorphism 
$A \xrightarrow{\sim} A \ot 1$, and
the twisted group algebra $\CC^\alpha G$ is a subalgebra of $A \#_\alpha G$ via
the isomorphism $\CC^\alpha G \xrightarrow{\sim} 1 \ot \CC^\alpha G$. 
The image of a group element $g$ in the twisted group algebra $\CC^\alpha G$ will
be denoted by $t_g$.
To shorten notation, we will write the element 
$a \ot g$ of $A \#_\alpha G$ by $at_g$.  
Since $\alpha$
is assumed to be normalized, $t_1$ is the multiplicative identity for $A \#_\alpha G$.
For all $g \in G$, we have 
\[
(t_g)^{-1} = \alpha^{-1}(g,g^{-1})t_{g^{-1}} = \alpha^{-1}(g^{-1},g)t_{g^{-1}}.
\]

Suppose that $G$ acts linearly on a complex vector space $V$ with basis $v_1,v_2,\ldots,v_n$, and
let $\q:= (q_{ij})_{1 \leq i,j \leq n}$ denote a tuple of
nonzero scalars for which $q_{ii}=1$ and $q_{ji}=q_{ij}^{-1}$ for all $i,j$.
For each group element $g\in G$, let $g^i_k$ denote the
scalar determined by the equation  
\[
   \lexp{g}{v_i} = \sum_{k=1}^n g_k^i v_k,
\]
and define the {\bf quantum $(i,j,k,l)$-minor determinant} of $g$ as 
\[
    \ddet_{ijkl} (g) := g^j_lg^i_k - q_{ji} g^i_lg^j_k.
\]

The following lemma will be used in the proof of Theorem~\ref{necessary and sufficient conditions} below.

\begin{lemma} 
\label{relation on det}
Suppose that the action of $G$ on $V$ extends to an action on $S_\q(V)$ by automorphisms,
and let $g \in G$. We have:
\begin{enumerate}
\item[(i)] $q_{lk} \ddet_{ijkl}(g) = - \ddet_{ijlk}(g)$ for all $i,j,k,l$.
\item[(ii)] For each $i,j$, if $q_{ij}\neq 1$, then $g^i_kg^j_k=0$ for all $k$.
\end{enumerate}
\end{lemma}
\begin{proof}
For a proof of part (i), see \cite[Lemma 3.2]{LS}. Part (ii) follows from
the assumption that $G$ acts on $S_\q(V)$ by automorphisms and that $q_{ij}\neq 1$:
We have $\lexp{g}{v_i}\lexp{g}{v_j} = q_{ij}\lexp{g}{v_j}\lexp{g}{v_i}$,
and so 
$\left( \sum_{k=1}^n g_k^i v_k \right)
\left( \sum_{l=1}^n g_l^j v_l \right)
=
q_{ij} \left( \sum_{k=1}^n g_k^j v_k \right)
\left( \sum_{l=1}^n g_l^i v_l \right)$.
Equating coefficients of $v_k^2$ yields $g_k^ig_k^j = q_{ij} g_k^i g_k^j$,
and since $q_{ij}\neq 1$, we get $g^i_kg^j_k=0$. 

\end{proof}

Let $\kappa: V\times V\rightarrow \CC^\alpha G$ be a bilinear map for which 
$\kappa(v_i,v_j) = - q_{ij}\kappa(v_j,v_i)$ for all $1\leq i,j\leq n$. 
For each $g\in G$, let $\kappa_g: V\times V\rightarrow \CC$ be the function determined by the
condition
\[
\kappa(v,w) = \sum_{g \in G} \kappa_g(v,w)t_g \qquad \text{ for all }  v, w \in V.
\]
The condition $\kappa(v_i,v_j) = -q_{ij}\kappa(v_j,v_i)$ 
implies that $\kappa_g(v_i,v_j) = -q_{ij}\kappa_g(v_j,v_i)$ for all $g \in G$.

Recall that the algebra 
\[
  \mH_{\q, \kappa, \alpha} := T(V)\#_\alpha G/(  v_iv_j - q_{ij} v_jv_i - \kappa(v_i,v_j) \mid
     1\leq i,j\leq n)
\]
is called a twisted quantum Drinfeld
Hecke algebra if it satisfies the Poincar\'e-Birkhoff-Witt condition:
$\gr \mH_{\q, \kappa, \alpha} \cong S_\q(V)\#_\alpha G $, as graded algebras.
This is equivalent to the condition that the set 
$\{v_1^{m_1}v_2^{m_2}\cdots v_n^{m_n}t_g \mid m_i \geq 0, g \in G\}$ 
is a $\CC$-basis for $\mH_{\q, \kappa, \alpha}$.

In the proof of the theorem below, we will assume that the reader is
familiar with G.~Bergman's 1978 paper on the Diamond Lemma \cite{B}.
We will freely use terminology (e.g. ``reduction system") defined in \cite{B}.

\begin{theorem} \label{necessary and sufficient conditions}
The algebra $\mH_{\q, \kappa, \alpha}$ is a twisted quantum Drinfeld Hecke algebra if and only if
the following conditions hold.
\begin{enumerate}
\item For all $g,h \in G$ and $1 \leq i < j \leq n$,
\[
\frac{\alpha(h,g)}{\alpha(hgh^{-1}, h)} \kappa_g(v_j, v_i)
= \sum_{k < l} \ddet_{ijkl}(h) \kappa_{hgh^{-1}}(v_l,v_k).
\]
\item For all $g \in G$ and $1 \leq i < j < k \leq n$,
\[
\kappa_g(v_k, v_j)(\lexp{g}{v_i} - q_{ji} q_{ki} v_i)
+ \kappa_g(v_k, v_i) (q_{kj} v_j - q_{ji} \lexp{g}{v_j}) 
+ \kappa_g(v_j, v_i)(q_{kj}q_{ki} \lexp{g}{v_k}-v_k)   = 0.
\]
\end{enumerate}
\end{theorem}

\begin{proof}
We begin by expressing the algebra $\mH_{\q, \kappa, \alpha}$
as a quotient of a free associative $\CC\text{-algebra}$.
Let $X=\{v_1,v_2, \ldots, v_n\} \cup \{t_g \mid g \in G\}$, and 
let $\CX$ be the free associative $\CC$-algebra generated by $X$.
Consider the reduction system 
\[
S=\{(t_gv_i, \lexp{g}{v_i}t_g), \, (t_gt_h, \alpha(g,h)t_{gh}), \, 
(v_jv_i, q_{ji}v_iv_j + \kappa(v_j,v_i)) \mid g,h \in G, 1 \leq i < j \leq n\}
\]
for $\CX$. Let $I$ be the ideal of $\CX$ generated by the following elements:
\[
t_gv_i -\lexp{g}{v_i}t_g, \qquad t_gt_h - \alpha(g,h) t_{gh}, \qquad 
v_jv_i-q_{ji}v_iv_j - \kappa(v_j,v_i), \qquad g,h \in G, 1 \leq i < j \leq n.
\]
In what follows, we will use the Diamond Lemma \cite{B} to show that the set 
\[
\{v_1^{m_1}v_2^{m_2}\cdots v_n^{m_n}t_g \mid m_i \geq 0, g \in G\}
\]
is a $\CC$-basis
for $\CX/I$ if and only if the two conditions in the statement of the theorem hold. 

Define a partial order $\leq$ on the free semigroup $\X$ as follow: First,
we declare that $v_1 < v_2 < \cdots < v_n < g$ for all $g \in G$, and then we set $A < B$ if 
\begin{enumerate}
\item[(i)] $A$ is of smaller length than $B$, or
\item[(ii)] $A$ and $B$ have the same length but $A$ is less than $B$ relative to
the lexicographical order.
\end{enumerate}
Then $\leq$ is a semigroup partial order on $\X$, compatible with the 
reduction system $S$, and having the descending chain condition. 
Thus, the hypothesis of the Diamond Lemma holds.

Observe that the set $\X_\irr$ of irreducible elements of $\X$ is precisely the
alleged $\CC$-basis for $\CX/I$. That is,
\[
\X_\irr = \{v_1^{m_1}v_2^{m_2}\cdots v_n^{m_n}t_g \mid m_i \geq 0, g \in G\}.
\]

In what follows, we show that all ambiguities of $S$ are
resolvable if and only if the two conditions in the statement of
the theorem hold. The theorem will then follow by the Diamond Lemma.
There are no inclusion ambiguities, but there exist overlap
ambiguities, and these correspond to the monomials 
\[
t_gt_ht_k, \qquad  t_gt_hv_i, \qquad t_hv_jv_i, \qquad v_kv_jv_i,  
\qquad \text{ where } 1 \leq i < j < k \leq n, \, g, h \in G.
\]

Associativity of the multiplication in the twisted group algebra $\CC^\alpha G$
implies that the ambiguities corresponding to the monomials $t_gt_ht_k$ is resolvable.
The equality $\lexp{gh}{v_i} = \lexp{g}{(\lexp{h}{v_i})}$
implies that the ambiguities corresponding to the monomials $t_gt_hv_i$ is resolvable.
Next, we show that the ambiguities corresponding to the monomials
$t_hv_jv_i$ is resolvable if and only if condition (1) in the statement of the
theorem holds. 
Applying a reduction to the factor $v_jv_i$ in $t_hv_jv_i$, we get
\[
q_{ji}t_hv_iv_j + t_h\kappa(v_j, v_i).
\]
Applying a reduction to the factor $t_hv_i$ and then to the resulting factor $t_hv_j$ gives
\[
\begin{aligned}[1]
q_{ji} \lexp{h}{v_i}\lexp{h}{v_j}t_h + t_h \kappa(v_j,v_i) 
&= q_{ji} \left(\sum_{l=1}^n h_l^i v_l \right) \left(\sum_{k=1}^n h_k^j v_k \right)t_h
+ t_h \kappa(v_j,v_i)\\
&= q_{ji} \sum_{l<k} h_l^ih_k^j v_lv_kt_h + q_{ji}\sum_{k<l} h_l^ih_k^j v_lv_kt_h 
+ q_{ji} \sum_{k=1}^n h_k^ih_k^jv_k^2t_h + t_h \kappa(v_j,v_i).\\
\end{aligned}
\]
Applying a reduction to the factor $v_lv_k$ in the second summation above yields
\[
q_{ji} \sum_{l<k} h_l^ih_k^j v_lv_kt_h + 
q_{ji}\sum_{k<l} h_l^ih_k^j q_{lk} v_k v_lt_h
+ q_{ji} \sum_{k<l}h_l^ih_k^j \kappa(v_l,v_k)t_h
+ \; q_{ji} \sum_{k=1}^n h_k^ih_k^jv_k^2t_h + t_h \kappa(v_j,v_i) \\
\]
Combining the first two summations, expanding  $\kappa(v_l,v_k)$ and $\kappa(v_j,v_i)$,
and then applying reductions to  each term in $\kappa(v_l,v_k)t_h$ and to each term in $t_h \kappa(v_j,v_i)$ gives
\[
\begin{aligned}[1]
& q_{ji} \sum_{k<l} \left(h_k^ih_l^j + q_{lk}h_l^ih_k^j\right) v_kv_lt_h 
+  q_{ji} \sum_{k=1}^n h_k^ih_k^jv_k^2t_h +  q_{ji} \sum_{g \in G} 
\left( \alpha(g,h) \sum_{k<l}h_l^ih_k^j \kappa_g(v_l,v_k)\right)t_{gh} \\
& \hspace{2.025in} \;+ \sum_{g \in G} \alpha(h,g) \kappa_g(v_j,v_i) t_{hg}\\
&= q_{ji} \sum_{k<l} \left(h_k^ih_l^j + q_{lk}h_l^ih_k^j\right) v_kv_lt_h 
+  q_{ji} \sum_{k=1}^n h_k^ih_k^jv_k^2t_h\\
& \hspace{1.05in} \;+ \sum_{g \in G} \left( \alpha(hgh^{-1},h) q_{ji} \sum_{k<l}h_l^ih_k^j \kappa_{hgh^{-1}}(v_l,v_k) 
+  \alpha(h,g) \kappa_g(v_j,v_i)\right) t_{hg}.
\end{aligned}
\]

Next, we apply to $t_hv_jv_i$ a reduction different from the one in the computation above:
Applying a reduction to the factor $t_hv_j$ in $t_hv_jv_i$, and then to the resulting factor $t_hv_i$, we get
\[
\lexp{h}{v_j}\lexp{h}{v_i}t_h 
= \left(\sum_{l=1}^n h_l^j v_l \right) \left(\sum_{k=1}^n h_k^i v_k \right)t_h 
= \sum_{l<k} h_l^jh_k^i v_lv_kt_h + \sum_{k<l} h_l^jh_k^i v_lv_kt_h 
+ \sum_{k=1}^n h_k^jh_k^iv_k^2t_h.
\]
Applying a reduction to the factor $v_lv_k$ in the second summation above yields
\[
\sum_{l<k} h_l^jh_k^i v_lv_kt_h + 
\sum_{k<l} q_{lk} h_l^jh_k^i v_kv_lt_h + \sum_{k<l}h_l^jh_k^i \kappa(v_l,v_k)t_h
+ \sum_{k=1}^n h_k^jh_k^iv_k^2t_h
\]
Combining the first two summations, expanding  $\kappa(v_l,v_k)$,
and then applying a reduction to each term in $\kappa(v_l,v_k)t_h$ gives
\[
\begin{aligned}[1]
&\sum_{k<l} \left(h_k^jh_l^i + q_{lk}h_l^jh_k^i\right) v_kv_lt_h 
+ \sum_{k=1}^n h_k^jh_k^iv_k^2t_h
+ \sum_{g \in G} \left( \alpha(g, h) \sum_{k<l}h_l^jh_k^i \kappa_g(v_l,v_k) \right) t_{gh}\\
&= \sum_{k<l} \left(h_k^jh_l^i + q_{lk}h_l^jh_k^i\right) v_kv_lt_h 
+ \sum_{k=1}^n h_k^jh_k^iv_k^2t_h
+ \sum_{g \in G} \left( \alpha(hgh^{-1}, h) \sum_{k<l}h_l^jh_k^i \kappa_{hgh^{-1}}(v_l,v_k) \right) t_{hg}.
\end{aligned}
\]

By equating coefficients, we see that the final expressions in the previous two 
computations are equal if and only if

\begin{enumerate}
\item[(a)] $q_{ji}h_k^ih_l^j + q_{ji}q_{lk}h_l^ih_k^j = h_k^jh_l^i + q_{lk}h_l^jh_k^i$ for all $k<l$,
\item[(b)] $q_{ji} h_k^ih_k^j = h_k^ih_k^j$ for all $k$, and
\item[(c)] for all $g \in G$, we have
\[
\ds \alpha(hgh^{-1},h) q_{ji} \sum_{k<l}h_l^ih_k^j \kappa_{hgh^{-1}}(v_l,v_k) 
+  \alpha(h,g) \kappa_g(v_j,v_i)\\
=\alpha(hgh^{-1}, h) \sum_{k<l}h_l^jh_k^i \kappa_{hgh^{-1}}(v_l,v_k).
\]
\end{enumerate}

Conditions (a) and (b) follow from part (i) and part (ii) of Lemma \ref{relation on det},
respectively. The equation in (c) is equivalent to condition (1) in the statement of the theorem.

Lastly, we show that the ambiguities corresponding to the monomials
$v_kv_jv_i$ is resolvable if and only if condition (2) in the statement of the
theorem holds. Applying a reduction to the factor $v_kv_j$ in 
$v_kv_jv_i$, we get
\[
q_{kj}v_jv_kv_i + \kappa(v_k,v_j)v_i.
\]
Applying a reduction to the factor $v_kv_i$ gives
\[
q_{kj}q_{ki}v_jv_iv_k + q_{kj}v_j\kappa(v_k,v_i) + \kappa(v_k,v_j)v_i.
\]
Applying a reduction to the factor $v_jv_i$ yields
\[
q_{kj}q_{ki}q_{ji}v_iv_jv_k + q_{kj}q_{ki}\kappa(v_j,v_i)v_k + q_{kj}v_j\kappa(v_k,v_i) 
+ \kappa(v_k,v_j)v_i.
\]
Expanding  $\kappa(v_j,v_i)$, $\kappa(v_k,v_i)$, and $\kappa(v_k,v_j)$,
applying reductions to each term in $\kappa(v_j,v_i)v_k$ and to each term in $\kappa(v_k,v_j)v_i$,
and then rearranging gives
\[
q_{kj}q_{ki}q_{ji}v_iv_jv_k 
+ \sum_{g \in G} \left( 
\kappa_g(v_k,v_j) \lexp{g}{v_i}
+ q_{kj} \kappa_g(v_k,v_i) v_j 
+ q_{kj}q_{ki}\kappa_g(v_j,v_i) \lexp{g}{v_k}
\right)t_g.
\]

Next, we apply to $v_kv_jv_i$ a reduction different from the one in the computation above:
Applying a reduction to the factor $v_jv_i$ in $v_kv_jv_i$, we get
\[
q_{ji}v_kv_iv_j + v_k\kappa(v_j,v_i).
\]
Applying a reduction to the factor $v_kv_i$ gives
\[
q_{ji}q_{ki} v_iv_kv_j + q_{ji} \kappa(v_k,v_i)v_j + v_k\kappa(v_j,v_i).
\]
Applying a reduction to the factor $v_kv_j$ yields
\[
q_{ji}q_{ki}q_{kj} v_iv_jv_k + q_{ji}q_{ki} v_i \kappa(v_k, v_j)
+ q_{ji} \kappa(v_k,v_i)v_j + v_k\kappa(v_j,v_i).
\]
Expanding  $\kappa(v_k,v_j)$, $\kappa(v_k,v_i)$, and $\kappa(v_j,v_i)$,
and then applying reductions to each term in $\kappa(v_k,v_i)v_j$ gives
\[
q_{ji}q_{ki}q_{kj}v_iv_jv_k + 
\sum_{g \in G} \left(
q_{ji}q_{ki}\kappa_g(v_k,v_j) v_i
+ q_{ji} \kappa_g(v_k,v_i) \lexp{g}{v_j} 
+ \kappa_g(v_j,v_i) v_k
\right)t_g.
\]
The final expressions in the two 
computations above are equal if and only if condition (2) in 
the statement of the theorem holds, finishing the proof.
\end{proof}

\end{section}
\begin{section}{Deformations}\label{deformations}

The primary goal of this section is to show that the twisted quantum Drinfeld Hecke algebras 
$\mH_{\q, \kappa, \alpha}$ are isomorphic to specializations of particular types of deformations
of the twisted skew group algebras $S_\q(V)\#_\alpha G$.  

Let $\hbar$ denote an indeterminate.
Recall that for a $\CC$-algebra $A$, a {\bf deformation of $A$ over $\CC[\hbar]$}
is an associative $\CC[\hbar]$-algebra whose underlying vector space is
$A[\hbar] = \CC[\hbar]\ot A$, and which reduces modulo $\hbar$ to the original algebra $A$.
Thus, the multiplication $\mu$ on $A[\hbar]$ is determined by
\[
  \mu(a,b) = \mu_0(a,b) + \mu_1(a,b) \hbar + \mu_2(a,b)\hbar^2 + \cdots
\]
for all $a,b\in A$, where $\mu_0(a,b)$ is the product in $A$, the
$\mu_i :A \times A \rightarrow A$ are $\CC$-bilinear maps extended to
be bilinear over $\CC[\hbar]$, and for each pair $(a,b)$ the sum above is finite.
A consequence of associativity of $\mu$ is that $\mu_1$ is a {\bf Hochschild 2-cocycle},
that is 
\begin{equation}\label{eqn:2-cocycle-condn}
a\mu_1(b, c) + \mu_1(a, bc) =  \mu_1(ab, c) + \mu_1(a, b)c 
\end{equation}
for all $a,b,c\in A$. 

In order to see that the twisted quantum Drinfeld Hecke algebras $\mH_{\q, \kappa, \alpha}$
may be realized as specializations of deformations of $S_\q(V)\#_\alpha G$,
we define the algebra
\[
  \mH_{\q, \kappa, \alpha, \hbar}:= (T(V)\#_\alpha G)[\hbar]/ ( v_iv_j - q_{ij} v_jv_i - \kappa(v_i,v_j)\hbar
    \mid 1\leq i,j\leq n).
\]

Assigning $\hbar$ degree zero, each $v_i$ degree one, and each $t_g$ ($g \in G$) degree zero, we see that
$\mH_{\q, \kappa, \alpha, \hbar}$ is a filtered algebra, and that 
$(S_\q(V)\#_\alpha G)[\hbar]$ is a graded algebra. We call the algebra $\mH_{\q, \kappa, \alpha, \hbar}$ 
a {\bf twisted quantum Drinfeld Hecke algebra over} $\CC[\hbar]$ if
$\gr \mH_{\q, \kappa,\alpha, \hbar} \cong (S_\q(V)\#_\alpha G)[\hbar]$, as graded algebras.
Specializing a twisted quantum Drinfeld Hecke algebra over $\CC[\hbar]$ to $\hbar=1$ yields the 
twisted quantum Drinfeld Hecke algebra over $\CC$, as defined earlier. 

In the theorem below, by the {\bf degree} of $\mu_i$, we mean its degree as a function
between graded algebras.

\begin{theorem}\label{main-thm}
Every twisted quantum Drinfeld Hecke algebra $\mH_{\q, \kappa, \alpha, \hbar}$ over $\CC[\hbar]$ is isomorphic 
to some deformation $\mu = \mu_0 + \mu_1 \hbar + \mu_2 \hbar^2 + \cdots$ of $S_\q(V)\#_\alpha G$ over $\CC[\hbar]$ 
with $\deg\mu_i = -2i$ for all $i\geq 1$.
\end{theorem}

\begin{proof}
Suppose that $\mH_{\q, \kappa, \alpha, \hbar}$ is a twisted quantum Drinfeld Hecke algebra over $\CC[\hbar]$. 
Consider the natural projection $T(V) \#_\alpha G \to S_\q(V) \#_\alpha G$, and let
$s: S_\q(V) \#_\alpha G \to T(V) \#_\alpha G$ be the $\CC\text{-linear section}$ 
determined by the ordering $v_1,v_2, \ldots, v_n$ of the basis of $V$. For example,
$s(v_2v_1^2t_g) = q_{21}^2 v_1^2v_2t_g$.

Extend $s$ to a $\CC[\hbar]$-linear map 
$\tilde{s}:(S_\q(V) \#_\alpha G)[\hbar] \to (T(V) \#_\alpha G)[\hbar]$, 
and let $p$ denote the natural projection from
$(T(V) \#_\alpha G)[\hbar]$ to $\mH_{\q, \kappa, \alpha, \hbar}$. 
Since $\mH_{\q, \kappa, \alpha, \hbar}$ is a twisted quantum Drinfeld Hecke algebra over $\CC[\hbar]$, 
the composition $f := p \circ \tilde{s}$ is an isomorphism of $\CC[\hbar]$-modules.

Next, define a $\CC[\hbar]$-bilinear multiplication $\mu$ on $(S_\q(V) \#_\alpha G)[\hbar]$
by 
\[
\mu := f^{-1} \circ \mult \circ (f \times f),
\]
where $\mult$ is the multiplication map in $\mH_{\q, \kappa, \alpha, \hbar}$. 
Since $\mu$ is $\CC[\hbar]$-bilinear, it must necessarily be a power
series 
\[
\mu = \mu_0 +  \mu_1 \hbar + \mu_2 \hbar + \cdots,
\]
where the $\mu_i$ are $\CC$-bilinear maps from  $(S_\q(V) \#_\alpha G) \times (S_\q(V) \#_\alpha G)$
to $S_\q(V) \#_\alpha G$. Note that, by definition of $f$, the map $\mu_0$ is precisely
the multiplication map in $S_\q(V) \#_\alpha G$, and so $\mu$ is a deformation $S_\q(V) \#_\alpha G$
over $\CC[\hbar]$.
By definition, the map $f$ is an isomorphism between the
$\CC[\hbar]$-algebras $(S_\q(V) \#_\alpha G[\hbar], \, \mu)$
and $\mH_{\q, \kappa, \alpha, \hbar}$, proving that $\mH_{\q, \kappa, \alpha, \hbar}$
is isomorphic to a deformation of $S_\q(V) \#_\alpha G$ over $\CC[\hbar]$.

Finally, we prove the degree condition on the $\mu_i$.
Given elements $a=v_1^{\beta_1}v_2^{\beta_2} \cdots v_n^{\beta_n}t_g$ and 
$b=v_1^{\gamma_1}v_2^{\gamma_2} \cdots v_n^{\gamma_n}t_h$
in $S_\q(V) \#_\alpha G$, to find $\mu_1(a,b), \mu_2(a,b), \ldots$,
we must put the product $f(a)f(b) \in \mH_{\q, \kappa, \alpha, \hbar}$ in the normal form by 
applying repeatedly the relations defining $\mH_{\q, \kappa, \alpha, \hbar}$.
Induction on the degree $\sum_{k=1}^n \beta_k + \gamma_k$ of $ab$ 
implies that $\deg\mu_i = -2i$ for all $i\geq 1$, as claimed.
\end{proof}

\begin{lemma}\label{TQDHA iff TQDHA over C[hbar]}
The algebra $\mH_{\q, \kappa, \alpha}$ is a twisted quantum Drinfeld Hecke algebra 
over $\CC$ if and only if $\mH_{\q, \kappa, \alpha, \hbar}$ is a 
twisted quantum Drinfeld Hecke algebra over $\CC[\hbar]$.
\end{lemma}
\begin{proof}
The proof given for $\mH_{\q, \kappa, \alpha}$  in Theorem~\ref{necessary and sufficient conditions} 
generalizes for $\mH_{\q, \kappa, \alpha, \hbar}$ by extending scalars to $\CC[\hbar]$.
That is, $\mH_{\q, \kappa, \alpha, \hbar}$ is a twisted quantum Drinfeld Hecke algebra
over $\CC[\hbar]$ if and only if the two conditions in Theorem~\ref{necessary and sufficient conditions} 
hold.
\end{proof}

\begin{corollary}
Every twisted quantum Drinfeld Hecke algebra $\mH_{\q, \kappa, \alpha}$
is isomorphic to a specialization of a deformation 
$\mu = \mu_0 + \mu_1 \hbar + \mu_2 \hbar^2 + \cdots$ of $S_\q(V)\#_\alpha G$ over $\CC[\hbar]$ 
with $\deg\mu_i = -2i$ for all $i\geq 1$.
\end{corollary}

A Hochschild $2$-cocycle on $S_\q(V)\#_\alpha G$
is said to be {\bf constant} if it is of degree $-2$ as a function between graded algebras.
In the next section, it is shown that such $2$-cocycles correspond to certain {\em constant} polynomials,
justifying the choice of terminology.

\begin{proposition} \label{quantum skew-symmetrization}
Let $\mH_{\q, \kappa, \alpha}$ be a twisted quantum Drinfeld Hecke algebra.
The map $\kappa: V \times V \to \CC^\alpha G$ is equal to the {\bf quantum skew-symmetrization}
of some constant Hochschild $2$-cocycle $\mu_1$ on $S_\q(V)\#_\alpha G$, that is, 
\[
\kappa(v_i,v_j) = \mu_1(v_i,v_j) - q_{ij} \mu_1(v_j, v_i)
\]
for all $i,j$.
\end{proposition}
\begin{proof}
By Lemma~\ref{TQDHA iff TQDHA over C[hbar]}, $\mH_{\q, \kappa, \alpha,\hbar}$ 
is a twisted quantum Drinfeld Hecke algebra over $\CC[\hbar]$.
By Theorem~\ref{main-thm}, associated to $\mH_{\q, \kappa, \alpha, \hbar}$ 
is a deformation $\mu = \mu_0 + \mu_1 \hbar + \mu_2 \hbar^2 + \cdots$ of $S_\q(V)\#_\alpha G$ over $\CC[\hbar]$ 
with $\deg\mu_i = -2i$ for all $i\geq 1$. Note that $\mu_1$ is a constant Hochschild $2$-cocycle
on  $S_\q(V)\#_\alpha G$. We claim that $\kappa$ is equal to the quantum skew-symmetrization
of $\mu_1$.

Let $f$ be the map defined in the proof of Theorem~\ref{main-thm}.
For any two monomials $a, b \in S_\q(V)\#_\alpha G$, the value
of $\mu_1(a,b)$ is determined by writing the product $f(a)f(b) \in \mH_{\q, \kappa, \alpha, \hbar}$
in the normal form by applying repeatedly the relations defining $\mH_{\q, \kappa, \alpha, \hbar}$. 
If $i \leq j$, then the
product $f(v_i)f(v_j)=v_iv_j$ is already in the desired form, so $\mu_1(v_i, v_j)=0$.
If $i>j$, then we write $v_iv_j  \longrightarrow q_{ij}v_jv_i + \kappa(v_i,v_j)\hbar$, and so
$\kappa(v_i,v_j)=\mu_1(v_i, v_j)$.
If $i \leq j$, we have $\kappa(v_i,v_j) = -q_{ij} \kappa(v_j,v_i) = -q_{ij} \mu_1(v_j,v_i)$.
Thus, $\kappa(v_i,v_j) = \mu_1(v_i,v_j) - q_{ij} \mu_1(v_j, v_i)$ for all $i,j$.
\end{proof}

The proof of the theorem below is a generalization of \cite[Theorem~2.2]{NW}; see also \cite[Theorem~3.2]{Wi}.

\begin{theorem}
Every deformation $\mu = \mu_0 + \mu_1 \hbar + \mu_2 \hbar^2 + \cdots$ of $S_\q(V)\#_\alpha G$ over $\CC[\hbar]$ 
with $\deg\mu_i = -2i$ for all $i\geq 1$ is isomorphic to some twisted quantum Drinfeld Hecke algebra over $\CC[\hbar]$.
\end{theorem}
\begin{proof}

Suppose that $\mu = \mu_0 + \mu_1 \hbar + \mu_2 \hbar^2 + \cdots$ is a deformation
of $S_\q(V)\#_\alpha G$ over $\CC[\hbar]$ with $\deg\mu_i = -2i$ for all $i\geq 1$.
In what follows, we will identity $T(V) \#_\alpha G$ with the free associative
$\CC$-algebra generated by the set $\{v_1, v_2, \ldots, v_n\} \cup \{t_g \mid g \in G\}$
subject to the relations $t_gv_i = \lexp{g}{v_i}t_g$ and $t_gt_h = \alpha(g,h)t_{gh}$
for all $i \in \{1,2, \ldots, n\}$, and all $g,h \in G$.
Define a map $\phi: (T(V) \#_\alpha G)[\hbar] \to (S_\q(V)\#_\alpha G)[\hbar]$
as follows. First, set $\phi(v_i)=v_i$ and $\phi(t_g)=t_g$ 
for all $i \in \{1,2, \ldots, n\}$, and all $g \in G$.
Since $\deg\mu_k = -2k$ for all $k\geq 1$, we have
\[
\mu_k(\CC^\alpha G, \CC^\alpha G) = \mu_k(\CC^\alpha G, V) = \mu_k(V, \CC^\alpha G) =0,
\]
for all $k \geq 1$. This implies that the relations 
$t_gv_i = \lexp{g}{v_i}t_g$ and $t_gt_h = \alpha(g,h)t_{gh}$
hold in the algebra $((S_\q(V)\#_\alpha G)[\hbar], \, \mu)$,
and so we obtain a $\CC$-algebra homomorphism on $T(V) \#_\alpha G$,
which extends to a $\CC[\hbar]$-algebra homomorphism $\phi$ from 
$(T(V) \#_\alpha G)[\hbar]$ to $(S_\q(V)\#_\alpha G)[\hbar]$,
where the algebra structure on the latter is given by $\mu$.

Next, we will show that $\phi$ is surjective. It is enough to
show that each monomial $v_{i_1}\cdots v_{i_m} t_g$ is in the
image of $\phi$. The proof is by induction on the degree of
the monomial. Suppose that all monomials of degree less than $m$
are in the image of $\phi$. So, in particular,
$\phi(X) = v_{i_2}\cdots v_{i_m} g$ for some $X\in (T(V)\#_\alpha G)[\hbar]$.
Then
\begin{eqnarray*}
   \phi(v_{i_1}X ) & = & \mu(v_{i_1}, \phi(X)) \\
  &=& \mu(v_{i_1}, v_{i_2}\cdots v_{i_m}t_g)\\
  &=& v_{i_1}\cdots v_{i_m} t_g + \mu_1(v_{i_1}, v_{i_2}\cdots v_{i_m}t_g)\hbar 
    + \mu_2(v_{i_1},v_{i_2}\cdots v_{i_m}t_g ) \hbar^2 + \cdots
\end{eqnarray*}
Since $\deg(\mu_k) = -2k$, by induction hypothesis, each $\mu_k(v_{i_1} , v_{i_2}\cdots
v_{i_m}t_g)$ is in the image of $\phi$. 
Therefore, $v_{i_1}\cdots v_{i_m}t_g$ is in the image of $\phi$, and it
follows that $\phi$ is surjective.

Finally, we determine the kernel of $\phi$.
Since $\deg(\mu_1) = -2$, we can define a 
bilinear map $\kappa: V \times V \to \CC^\alpha G$
by setting $\kappa(v_i, v_j) := 
\mu_1(v_i, v_j)-q_{ij}\mu_1(v_j, v_i)$
for all $i,j$. Let $I$ denote the ideal in $(T(V) \#_\alpha G)[\hbar]$
generated by the elements 
\[
v_iv_j - q_{ij} v_jv_i - \kappa(v_i,v_j)\hbar.
\]

Since $\mu_k(v_i,v_j)=0$ for all $k \geq 2$, we have
\begin{eqnarray*}
  \phi(v_iv_j) & = & \mu(v_i, v_j) \ \ = \ \ v_iv_j + \mu_1(v_i, v_j) \hbar,\\
  \phi(v_jv_i) & = & \mu(v_j,v_i) \ \ = \ \  v_jv_i + \mu_1(v_j, v_i)\hbar,
\end{eqnarray*}
and so $I$ is contained in the kernel of $\phi$.
The form of the relations and surjectivity of $\phi$ imply that 
the kernel of $\phi$ is precisely $I$, and it follows that the
deformation $((S_\q(V)\#_\alpha G)[\hbar], \, \mu)$ is isomorphic
to the twisted quantum Drinfeld Hecke algebra $\mH_{\q, \kappa, \alpha, \hbar}$
over $\CC[\hbar]$.

\end{proof}

\end{section}
\begin{section}{Computing $\HH^2(S_\q(V)\#_\alpha G)$} \label{computing HH^2}

Let $A$ be an algebra on which the finite group $G$ acts by automorphisms,
and let $\alpha$ be a $2\text{-cocycle}$ on $G$. This section is concerned with the 
computation of the Hochschild cohomology
$\HHD(A\#_\alpha G)$ of the twisted skew group algebra $A\#_\alpha G$.
We will be particularly interested in degree two cohomology in the case when $A$ is 
the quantum symmetric algebra $S_\q(V)$. The results of this section
will be used in the sections that follow.

Recall that the Hochschild cohomology of an algebra $R$ is $\HHD(R):= \Ext^{\bu}_{R^e}(R,R)$,
where the enveloping algebra $R^e :=R\ot R^{op}$ acts on $R$ by left and right multiplication. 
When $R$ is a twisted skew group algebra $A\#_\alpha G$ in a characteristic not dividing the
order of the finite group $G$,
by a result of \c{S}tefan \cite[Corollary 3.4]{St}, there is an action of $G$ on $\HHD(A,A\#_\alpha G)= 
\Ext^{\bu}_{A^e}(A,A\#_\alpha G)$ for which $\HHD(A\#_\alpha G)$ is isomorphic
to $\HHD(A,A\#_\alpha G)^G$, the space of elements of $\HHD(A, A\#_\alpha G)$ that are invariant under 
the action of $G$.
Thus, one can compute $\HHD(A\#_\alpha G)$ by first computing $\HHD(A,A\#_\alpha G)$
and then determining the space of $G$-invariant elements. When $A$ is the 
the quantum symmetric algebra $S_\q(V)$, we compute 
$\HHD(S_\q(V),S_\q(V)\#_\alpha G)$ using the quantum Koszul resolution,
recalled below.

The {\bf quantum exterior algebra} $\Wedge _\q(V)$ associated
to the tuple $\q = (q_{ij})$ is
\[
  \Wedge_\q(V):= \CC\langle v_1,\ldots,v_n\mid v_iv_j = -q_{ij}v_jv_i
   \mbox{ for all } 1\leq i,j\leq n\rangle.
\]
Since we are working in characteristic 0, 
the defining relations imply in particular that 
$v_i^2=0$ for each $v_i$ in $\Wedge_\q(V)$.
This algebra has a basis given by all $v_{i_1}\cdots v_{i_m}$ ($0\leq m\leq n$,
$1\leq i_1<\cdots <i_m\leq n$); we will write such a basis element as $v_{i_1}\wedge
\cdots \wedge v_{i_m}$ by analogy with the ordinary exterior algebra.

By \cite[Proposition 4.1(c)]{W},
the following is a free $S_\q(V)^e$-resolution of $S_\q(V)$:

\begin{equation}
\label{label: free resolution}
\cdots \xrightarrow{} S_\q(V)^e\ot\Wedge_\q^2(V) \xrightarrow{d_2}
S_\q(V)^e \otimes \Wedge_\q ^1(V) \xrightarrow{d_1}
S_\q(V)^e \xrightarrow{\text{mult}} S_\q(V) \xrightarrow{} 0,
\end{equation}
that is, for $1\leq m\leq n$, the degree $m$ term is $S_\q(V)^e\otimes \Wedge_\q^m(V)$;
the differential $d_m$  is defined by 
\begin{equation*}
\begin{split}
&d_m(1^{\ot 2}\ot v_{j_1}\wedge\cdots\wedge v_{j_m})\\
&=  \sum_{i=1}^m (-1)^{i+1} \left[ \left(\prod_{s=1}^{i} q_{j_s, j_i} \right)
    v_{j_i}\ot 1 - \left(\prod_{s=i}^m q_{j_i, j_s} \right) \ot v_{j_i} \right] \ot
    v_{j_1}\wedge \cdots \wedge \hat{v}_{j_i} \wedge\cdots \wedge v_{j_m}
\end{split}
\end{equation*}
whenever $1\leq j_1 <\ldots < j_m\leq n$, and mult denotes the multiplication
map. The complex \eqref{label: free resolution} 
is a quantum version of the usual Koszul resolution for a polynomial ring.

Suppose that the action of $G$ on $V$ induces an action on $\Wedge_\q(V)$.
Thus, there is an action of $G$ on the quantum Koszul complex \eqref{label: free resolution},
that is, an action of $G$ on each $S_\q(V)^e\ot \Wedge^i_\q(V)$ that commutes with
the differentials.

Assume that $\HHD(S_\q(V)\#_\alpha G)$ 
has been computed using the quantum Koszul resolution. So, elements of $\HHD(S_\q(V)\#_\alpha G)$
are given as $G$-invariant elements of $\HHD(S_\q(V), S_\q(V)\#_\alpha G)$.
For our purposes, we will need to find representatives for elements in $\HH^2(S_\q(V)\#_\alpha G)$
that are given as maps from $(S_\q(V)\#_\alpha G) \ot (S_\q(V)\#_\alpha G)$
to $S_\q(V)\#_\alpha G$ satisfying the $2$-cocycle condition \eqref{eqn:2-cocycle-condn}.
To this end, we consider chain maps between the quantum Koszul resolution \eqref{label: free resolution}
and the bar resolution of $A$:
\[
\xymatrix{
\cdots \ar[r] & S_\q(V)^{\ot 4}\ar[r]^{\delta_2}\ar@<-2pt>[d]_{\Psi_2} 
               & S_\q(V)^{\ot 3}\ar[r]^{\delta_1}\ar@<-2pt>[d]_{\Psi_1} 
               & S_\q(V)^e \ar[r]^{\text{mult}} \ar[d]_{=}
               & S_\q(V) \ar[r] \ar[d]_{=} & 0\\
\cdots \ar[r] & S_\q(V)^e\ot \Wedge_\q ^2 V \ar[r]^{d_2}\ar@<-2pt>[u]_{\Phi_2}
               & S_\q(V)^e\ot \Wedge_\q ^1 V \ar[r]^{\hspace{.6cm}d_1}\ar@<-2pt>[u]_{\Phi_1}
               & S_\q(V)^e \ar[r]^{\text{mult}} \ar[u]
               & S_\q(V)\ar[r] \ar[u] & 0 .
}
\]
Here the differentials $\delta_i$ in the bar resolution are defined as 
\[
\delta_i(a_0\ot\cdots \ot a_{i+1}) = \sum_{j=0}^i (-1)^j a_0\ot\cdots \ot a_j a_{j+1}
  \ot\cdots\ot a_{i+1}
\]
for all $a_0,\ldots,a_{i+1}\in A$.
We will only need to know the values of $\Psi_2$ on elements of the form
$1\ot v_i\ot v_j\ot 1$, and these can be chosen to be
\begin{equation}\label{psi-two}
  \Psi_2(1\ot v_i\ot v_j\ot 1) = 
\begin{cases}
1\ot 1\ot v_i\wedge v_j & \text{ if } i<j, \\
0 & \text{ if } i \geq j.
\end{cases}
\end{equation}

Chain maps $\Phi_i$ are defined in \cite{NSW}, and more generally in \cite{W}, 
that embed the quantum Koszul resolution
as a subcomplex of the bar resolution. We will only need $\Phi_2$, and this is
defined by

\begin{equation}\label{phim}
  \Phi_m(1 \ot 1 \ot v_i \wedge \wedge v_j)
  = 1 \ot v_i \ot v_j \ot 1 - q_{ij} \ot v_j \ot v_i \ot 1
\end{equation}
for all $1 \leq i,j \leq n$.

We define the Reynold's operator, or averaging map, which ensures
$G$-invariance of the image, compensating for the possibility that 
$\Psi_2$ may not preserve the action of $G$:
\begin{eqnarray*}
\R_2: \Hom_{\CC}(S_\q(V)^{\ot 2}, S_\q(V) \#_\alpha G) &\to & 
\Hom_{\CC}(S_\q(V)^{\ot 2}, S_\q(V) \#_\alpha G)^G\\
\R_2(\gamma) &:=  &\frac{1}{|G|} \sum_{g \in G} \lexp{g}{\gamma}.
\end{eqnarray*}
A map that tells how to extend a function defined on
$S_\q(V)^{\ot 2}$ to a function defined on $(S_\q(V)\#_\alpha G)^{\ot 2}$ is
from \cite{CGW}: 
\begin{eqnarray*}
\Theta_2^*:  \Hom_{\CC}(S_\q(V)^{\ot 2}, S_\q(V) \#_\alpha G)^G &\to & 
\Hom_{\CC}((S_\q(V) \#_\alpha G)^{\ot 2}, S_\q(V) \#_\alpha G)\\
\Theta_2^*(\kappa)(a_1 t_{g_1} \ot a_2  t_{g_2}) & := & \alpha(g_1,g_2) \kappa(a_1 \ot \lexp{g_1}{a_2})t_{g_1g_2}.
\end{eqnarray*}

The theorem below is from \cite{CGW}; see also \cite{SW}.

\begin{theorem}[\cite{CGW}]\label{thm: CGW}
Suppose that the action of $G$ on $V$ extends to an action on $\Wedge_\q(V)$ by automorphisms.
The map
\[
 \Theta_2^* \R_2  \Psi_2^* : \Hom_\CC \left(\Wedge_\q^2(V), S_\q(V) \#_\alpha G \right) 
\to \Hom_\CC \left(S_\q(V)^{\ot 2}, S_\q(V) \#_\alpha G \right) 
\]
induces an isomorphism
\[
\HH^{2} (S_\q(V), S_\q(V) \#_\alpha G)^G 
\xrightarrow{\sim} 
\HH^2(S_\q(V) \#_\alpha G).
\]
Moreover, $\Theta_2^* \R_2  \Psi_2^* $ maps
$\HH^{2}(S_\q(V), S_\q(V) \#_\alpha G)$ onto $\HH^2(S_\q(V) \#_\alpha G)$.
\end{theorem}

Next, we will introduce some
notation, and give some formulas that will be
useful in the sections that follow. For each $g\in G$, the
space $S_\q(V)t_g \subseteq S_\q(V) \#_\alpha G$ is a (left) $S_\q(V)^e$-module via the action 
\[
   (a\ot b) \cdot (c t_g) := ac t_g b = ac (\lexp{g}{b}) t_g
\]
for all $a,b,c\in S_\q(V)$, and all $g\in G$.
Note that $\HH^{2} (S_\q(V), S_\q(V) \#_\alpha G)$
is isomorphic to the direct sum $\bigoplus_{g \in G} \HH^{2} (S_\q(V), S_\q(V)t_g)$.

We wish to express the formula for the differentials $d_m$ in the quantum koszul resolution
\eqref{label: free resolution} in a more convenient form. 
To this end, let $\N^n$ denote the set of all $n$-tuples of elements
from $\N$. The {\bf length} of $\gamma = (\gamma_1,\ldots,\gamma_n) \in \N^n$, 
denoted $|\gamma|$, is the sum $\sum_{i=1}^n \gamma_i$.
For each $\gamma \in \N^n$, define $v^\gamma := 
v_1^{\gamma_1} v_2^{\gamma_2} \cdots v_n^{\gamma_n}$. For each 
$i \in \{1, \ldots, n\}$,  define $[i] \in \N^n$ by
$[i]_j = \delta_{i,j}$, for all $j \in \{1, \ldots, n\}$.
For each $\beta=(\beta_1,\ldots,\beta_n) \in \O^n$, let $v^{\wedge \beta}$ denote
the vector $v_{j_1} \wedge \cdots \wedge v_{j_m} \in \Wedge_\q^m(V)$
determined by the conditions $m = |\beta|$, $\beta_{j_k} = 1$ for all 
$k \in \{1, \ldots, m\}$, and $j_1<\ldots <j_m$.
For each $\beta \in \O^n$ with $|\beta| = m$, we have
\begin{equation*}
d_m(1^{\ot 2}\ot v^{\wedge \beta}) =
\sum_{i=1}^n \delta_{\beta_i,1} (-1)^{\sum_{s=1}^{i-1} \beta_s} 
\left[ \left(\prod_{s=1}^i q_{s, i}^{\beta_s} \right)
    v_i\ot 1 - \left(\prod_{s=i}^n q_{i, s}^{\beta_s} \right) \ot v_i \right] \ot
    v^{\wedge(\beta-[i])}.
\end{equation*}

Removing the term $S_\q(V)$ from the quantum koszul resolution \eqref{label: free resolution},
applying the functor $\Hom_{S_\q(V)^e}(\,\cdot, \, S_\q(V)t_g)$,
and then identifying $\Hom_{S_\q(V)^e}\left(S_\q(V)^e \ot \Wedge_\q^{\bu}(V), S_\q(V)t_g\right)
\cong \Hom_\CC\left(\Wedge_\q^{\bu}(V), S_\q(V)t_g\right)$ with 
$S_\q(V)t_g \ot \Wedge_{\q^{-1}}^{\bu}(V^*)$, we obtain complex
\begin{equation}
\label{new Hom(resolution) with G}
0 \xrightarrow{}  S_\q(V)t_g \xrightarrow{d_1^*} 
S_\q(V)t_g \ot \Wedge^1_{\q^{-1}}(V^*) \xrightarrow{d_2^*} 
S_\q(V)t_g \ot \Wedge^2_{\q^{-1}}(V^*) \xrightarrow{} \cdots.
\end{equation}
For all $a \in S_\q(V)$,  and all $\beta \in \O^n$ with $|\beta|=m-1$,
the differential $d_m^*$ sends the element $a t_g\ot (v^*)^{\wedge\beta}$ to
\begin{equation}
\label{formula for d_m^* with G}
\sum_{i=1}^n
\delta_{\beta_i,0} (-1)^{\sum_{s=1}^i \beta_s} \left[
\left( \left( \prod_{s=1}^i q_{s,i}^{\beta_s} \right) v_ia - 
\left( \prod_{s=i}^n q_{i,s}^{\beta_s}\right) a(\lexp{g}{v_i}) \right) t_g \right]
\ot {(v^*)}^{\wedge(\beta+[i])}.
\end{equation}

For later use, we record the following formula.
Let $\eta \in( S_\q(V) \#_\alpha G) \ot \Wedge_{\q^{-1}}^2(V^*)$.
Then,
\begin{equation}
\label{composition}
[\Theta_2^* \R_2   \Psi_2^* (\eta)]
(v_i  \ot v_j ) = \frac{1}{|G|} \sum_{g \in G} \,  \lexp{g}\!
  {(\eta (\Psi_2(1 \ot \lexp{g^{-1}}{v_i} \ot \lexp{g^{-1}}{v_j} \ot 1)))}.
\end{equation}

The elements of $((S_\q(V)\#_\alpha G)\ot \Wedge_{\q^{-1}}^2(V^*))^G$ that correspond to 
{\em constant} Hochschild two-cocycles,
that is, those of degree $-2$ as maps from $(S_\q(V)\#_\alpha G)\ot (S_\q(V)\#_\alpha G)$ to $S_\q(V)\#_\alpha G$, 
are precisely those in $(\CC^\alpha G\ot\Wedge_{\q^{-1}}^2(V^*))^G$, due to the form of the chain map $\Psi_2$.
Note that the intersection of the 
image of $d_2^*$ with $\CC^\alpha G\ot\Wedge_{\q^{-1}}^2(V^*)$ is 0.
Applying our earlier formula, letting $\beta = [j] + [k]$, 

\begin{equation}  \label{formula for d_3^*}
d_3^*( t_g\ot v_j^*\wedge v_k^*) 
  = \sum_{i\not\in\{j,k\}} (-1)^{\sum_{s=1}^i \beta_s} 
   \left[ \left(\left( \prod_{s=1}^i q_{s,i}^{\beta_s} \right) 
    v_i - \left(\prod _{s=i}^n q_{i,s}^{\beta_s}\right)
    \lexp{g}{v_i} \right)  t_g \right] \ot (v^*)^{\wedge(\beta + [i])}.  
\end{equation}

\end{section}
\begin{section}{Constant Hochschild $2$-cocycles}\label{constant cocycles}

In this section, we will establish the following bijection:
\[
\left\{
\begin{tabular}{c}
\text{constant Hochschild} \\
\text{ $2$-cocycles on } $S_\q(V)\#_\alpha G$
\end{tabular}
\right\}
\longleftrightarrow
\left\{
\begin{tabular}{c}
\text{twisted quantum Drinfeld } \\
\text{Hecke algebras } $\mH_{\q, \kappa, \alpha}$
\end{tabular}\right\}.
\]
We will also show that every constant Hochschild $2$-cocycles on $S_\q(V)\#_\alpha G$
lifts to a deformation of $S_\q(V)\#_\alpha G$.

We will use the following two lemmas shortly.

\begin{lemma}
The action of $G$ on $V$ extends to an action on $\Wedge_\q(V)$ by automorphisms
if, and only if, for all $g\in G$, $i\neq j$ and $k<l$,
\[
   (1-q_{ij}q_{lk} ) g^i_kg^j_l + (q_{ij} - q_{lk}) g^i_lg^j_k=0.
\]
\end{lemma}

\begin{proof}
See \cite[Lemma 4.2]{NW}.
\end{proof}

\begin{lemma}\label{two-actions}
Suppose that the action of $G$ on $V$ extends to an action, by automorphisms,
on $S_\q(V)$ and on $\Wedge_\q(V)$.
Then for all $g\in G$ and all $i,j,k,l$ ($i<j$, $k<l$), if $g^i_lg^j_k\neq 0$ then
$q_{lk}=q_{ij}$, and if $g^i_kg^j_l \neq 0$, then $q_{lk}=q_{ij}^{-1}$.
\end{lemma}

\begin{proof}
See \cite[Lemma 4.3]{NW}.
\end{proof}

Proposition~\ref{quantum skew-symmetrization} showed that every twisted 
quantum Drinfeld Hecke algebra arises from the quantum skew-symmetrization of 
a constant Hochschild $2$-cocycles. 
The theorem below shows that the converse is also true. The proof of the
following theorem involves the maps $\Theta_2^*, \R_2$, $\Psi_2^*$,
and $d_3^*$ defined in Section~\ref{computing HH^2}.

\begin{theorem} \label{thm: constant}
Suppose that the action of $G$ on $V$ extends to an action, by automorphisms, 
on $S_\q(V)$ and on $\Wedge_\q(V)$.
Let $\alpha$ be a normalized $2$-cocycle on $G$, 
let $\mu_1$ be a constant Hochschild 2-cocycle on $S_\q(V)\#_\alpha G$, 
and let $\kappa : V \times V \to \CC^\alpha G$ be the quantum skew-symmetrization of $\mu_1$.
Then, $\mH_{\q, \kappa, \alpha}$ is a twisted quantum Drinfeld Hecke algebra.
\end{theorem}

\begin{proof}
We will show that the map $\kappa$ satisfies the conditions of 
Theorem~\ref{necessary and sufficient conditions}. Let $\eta$ be a 
$G$-invariant element of 
\[
\Hom_\CC\left(\Wedge_\q^2(V), S_\q(V) \#_\alpha G\right) \cong (S_\q(V) \#_\alpha G) \ot \Wedge^2_{\q^{-1}}(V^*).
\]
such that $[\Theta_2^* \R_2 \Psi_2^*](\eta)=\mu_1$.
Since $\mu_1$ is a {\em constant} Hochschild $2$-cocycle, the image of $\eta$ as a map
from $\Wedge_\q^2(V)$ to $S_\q(V) \#_\alpha G$ is contained in $\CC^\alpha G$, equivalently,
$\eta$ belongs to $(\CC^\alpha G) \ot \Wedge^2_{\q^{-1}}(V^*)$.

For all $1 \leq k, l \leq n$, we have
$[\Psi_2^*(\eta)](v_k \ot v_l - q_{kl} v_l \ot v_k) = \eta(v_k \wedge v_l)$.
This equality and the $G$-invariance of $\eta$ imply that
$\kappa(v_i,v_j) = \eta(v_i \wedge v_j)$ for all $1 \leq i,j \leq n$.
Indeed, we have

\begin{eqnarray*}
\kappa(v_i,v_j) &=& [\Theta_2^* \R_2 \Psi_2^*(\eta)](v_i \ot v_j - q_{ij} v_j \ot v_i) \\
&=& \frac{1}{|G|} \sum_{g \in G} \Theta_2^*(\lexp{g}{(\Psi_2^*(\eta))})(v_i \ot v_j - q_{ij} v_j \ot v_i) \\
&=& \frac{1}{|G|} \sum_{g \in G} \lexp{g}{\left((\Psi_2^*(\eta)) \lexp{g^{-1}}{(v_i \ot v_j - q_{ij} v_j \ot v_i)}\right)} \\
&=& \frac{1}{|G|} \sum_{g \in G} \lexp{g}{\left((\Psi_2^*(\eta)) \left(\sum_{k,l} 
(g^{-1})^i_k (g^{-1})^j_l (v_k \ot v_l - q_{ij} v_l \ot v_k)\right) \right)}  \\
&=& \frac{1}{|G|} \sum_{g \in G} \lexp{g}{\left(\sum_{k,l} 
(g^{-1})^i_k (g^{-1})^j_l(\Psi_2^*(\eta))  (v_k \ot v_l - q_{ij} v_l \ot v_k) \right)}  \\
&=& \frac{1}{|G|} \sum_{g \in G} \lexp{g}{\left(\sum_{k,l} 
(g^{-1})^i_k (g^{-1})^j_l  \eta(v_k \wedge v_l) \right)}  \\
&=& \frac{1}{|G|} \sum_{g \in G} \lexp{g}{\left(
\eta ( \lexp{g^{-1}}{(v_i \wedge v_j)} ) \right)}  \\
&=& \frac{1}{|G|} \sum_{g \in G} (\lexp{g}{\eta}) (v_i \wedge v_j)\\
&=& \eta(v_i \wedge v_j).
\end{eqnarray*}

Next, write
\[
  \eta = \sum_{g\in G}\sum_{1\leq r<s\leq n} \eta^g_{rs}t_g\ot v_r^*\wedge v_s^*
\in \CC^\alpha G \ot \Wedge_{\bf q^{-1}}^2(V^*) \subseteq (S_\q(V)\#_\alpha G) \ot \Wedge_{\bf q^{-1}}^2(V^*).
\]
The calculation above implies that $\kappa_g(v_i,v_j) = \eta^g_{ij}$ for all $i < j$, and all $g \in G$.
Since $\eta$ is a Hochschild $2$-cocycle, we have $d_3^*(\eta)=0$. Using \eqref{formula for d_3^*}, we see 
that, for all $1 \leq i < j < k \leq n$, we must have
\[
\sum_{g\in G} (\eta^g_{jk}v_i t_g - \eta^g_{jk} q_{ij}q_{ik}\lexp{g}{v_i}t_g
   -\eta^g_{ik}q_{ij}v_jt_g + \eta^g_{ik} q_{jk}\lexp{g}{v_j}t_g 
    + \eta^g_{ij}q_{ik}q_{jk}v_kt_g -\eta^g_{ij}\lexp{g}{v_k}t_g)=0.
\]

Equivalently,

\[
  -\eta^g_{jk}(q_{ij}q_{ik} \lexp{g}{v_i} - v_i) -\eta_{ik}^g(q_{ij}v_j-q_{jk} \lexp{g}{v_j})
   - \eta^g_{ij}(\lexp{g}{v_k} - q_{ik}q_{jk}v_k)=0
\]
for all $1 \leq i < j < k \leq n$, and all $g\in G$. 

Multiplying both sides by $q_{ji}q_{ki}q_{kj}$ yields

\[
-q_{jk} \eta_{jk}^g(\lexp{g}{v_i} - q_{ji} q_{ki} v_i)
-q_{ki} \eta_{ik}^g (q_{kj} v_j - q_{ji} \lexp{g}{v_j}) 
-q_{ji} \eta_{ij}^g(q_{kj}q_{ki} \lexp{g}{v_k}-v_k)   = 0.
\]

Substituting $\kappa_g(v_k,v_j), \kappa_g(v_k,v_i)$, and $\kappa_g(v_j,v_i)$
for $-q_{jk} \eta_{jk}^g, -q_{ki} \eta_{ik}^g$,  and $-q_{ji} \eta_{ij}^g$,
respectively, we obtain

\[
\kappa_g(v_k, v_j)(\lexp{g}{v_i} - q_{ji} q_{ki} v_i)
+ \kappa_g(v_k, v_i) (q_{kj} v_j - q_{ji} \lexp{g}{v_j}) 
+ \kappa_g(v_j, v_i)(q_{kj}q_{ki} \lexp{g}{v_k}-v_k)   = 0,
\]
which is condition (2) of Theorem~\ref{necessary and sufficient conditions}.

Next, we will show that $\kappa$ also satisfies condition (1) of 
Theorem~\ref{necessary and sufficient conditions}.
Since $\eta$ is $G$-invariant, we have 
$\eta(\lexp{h}{v_i} \wedge  \lexp{h}{v_j}) = \lexp{h}{(\eta(v_i\wedge v_j))}$ for all $i,j$, and all $h\in G$.
We have
\begin{eqnarray*}
  \eta( \lexp{h}{v_i} \wedge  \lexp{h}{v_j}) & = & \sum_{k,l} h^i_kh^j_l \eta(v_k\wedge v_l)\\
           &=& \sum_{k<l} h^i_kh^j_l\eta(v_k\wedge v_l) - \sum_{k<l} q_{lk}
   h^i_lh^j_k \eta(v_k\wedge v_l)\\
    &=& \sum_{k<l, \  g\in G} (h^i_kh^j_l - q_{lk}h^i_lh^j_k) \eta^g_{kl} t_g,
\end{eqnarray*}
and for all $i<j$, we have
\begin{eqnarray*}
  \lexp{h}{(\eta(v_i\wedge v_j))} &=& \lexp{h}{\left(\sum_{g\in G} \eta^g_{ij}t_g\right)} \\
&=& \sum_{g\in G} \eta^g_{ij} t_ht_g(t_h)^{-1} \\
&=& \sum_{g\in G} \frac{\alpha(h,g)\alpha(hg, h^{-1})}{\alpha(h^{-1},h)} \eta^g_{ij} t_{hgh^{-1}} \\
&=& \sum_{g\in G} \frac{\alpha(h,g)}{\alpha(hgh^{-1},h)} \eta^g_{ij} t_{hgh^{-1}}.
\end{eqnarray*}
Equating the coefficients of $t_{hgh^{-1}}$, we find that, for all $i<j$ and all $h,g \in G$, we have
\[
  \frac{\alpha(h,g)}{\alpha(hgh^{-1},h)} \eta^g_{ij} = \sum_{k<l} (h^i_kh^j_l-q_{lk}h^i_lh^j_k) \eta^{hgh^{-1}}_{kl}.
\]
Substituting $\kappa_g(v_i,v_j)$ and $\kappa_{hgh^{-1}}(v_k,v_l)$ for $\eta^g_{ij}$ and $\eta^{hgh^{-1}}_{kl}$,
respectively, and then multiplying both sides by $-q_{ji}$ yields
\[
  \frac{\alpha(h,g)}{\alpha(hgh^{-1},h)} \kappa_g(v_j,v_i) = 
\sum_{k<l} (q_{ji}q_{lk}h^i_lh^j_k - q_{ji}h^i_kh^j_l) \kappa_{hgh^{-1}}(v_k,v_l).
\]
Substituting $-q_{kl}\kappa_{hgh^{-1}}(v_l,v_k)$ for $\kappa_{hgh^{-1}}(v_k,v_l)$, and then
using Lemma~\ref{two-actions}, we obtain
\[
\frac{\alpha(h,g)}{\alpha(hgh^{-1}, h)} \kappa_g(v_j, v_i)
= \sum_{k < l} \ddet_{ijkl}(h) \kappa_{hgh^{-1}}(v_l,v_k),
\]
which is condition (1) of Theorem~\ref{necessary and sufficient conditions}.

\end{proof}

The proof of the following theorem involves the map $\Phi_2^*$
defined in Section~\ref{computing HH^2}.

\begin{theorem} \label{thm: bijection}
Let $\alpha$ be a normalized $2$-cocycle on $G$.
Suppose that the action of $G$ on $V$ extends to an action, by automorphisms, 
on $S_\q(V)$ and on $\Wedge_\q(V)$.
The assignment 
\[
\mu_1  \mapsto \mH_{\q, \kappa, \alpha}
\]
where $\kappa$ is the quantum skew-symmetrization of $\mu_1$ is a
bijection from the space of equivalences classes of constant Hochschild 2-cocycles on $S_\q(V)\#_\alpha G$
to the space of twisted quantum Drinfeld Hecke algebras associated to the quadruple $(G, V, \q, \alpha)$.
\end{theorem}

\begin{proof}
Proposition~\ref{quantum skew-symmetrization} showed that the assignment
specified in the statement of the theorem is surjective. To see that
the assignment is also injective, let $\mu_1$ and  $\mu_1'$ be constant Hochschild $2$-cocycles 
on $S_\q(V)\#_\alpha G$ such that their quantum skew-symmetrizations are equal.
We have
\begin{eqnarray*}
[\Phi_2^*(\mu_1)](1 \ot 1 \ot v_i \ot v_j) &=& \mu_1(v_i,v_j) - q_{ij} \mu_1(v_j,v_i)\\
&=& \mu_1'(v_i,v_j) - q_{ij} \mu_1'(v_j,v_i)\\
&=& [\Phi_2^*(\mu_1')](1 \ot 1 \ot v_i \ot v_j),
\end{eqnarray*}
so $\Phi_2^*(\mu_1) = \Phi_2^*(\mu_1')$, and it follows that 
$\mu_1$ and $\mu_1'$ are cohomologous.
\end{proof}

We end this section with the following.

\begin{theorem}
\label{thm: constant lifts}
Let $\alpha$ be a normalized $2$-cocycle on $G$.
Suppose that the action of $G$ on $V$ extends to an action, by automorphisms, 
on $S_\q(V)$ and on $\Wedge_\q(V)$.
Each constant Hochschild 2-cocycle on $S_\q(V)\#_\alpha G$
lifts to a deformation of $S_\q(V)\#_\alpha G$ over $\CC[\hbar]$. 
\end{theorem}

\begin{proof}
Let $\mu_1'$ be a constant Hochschild 2-cocycle on $S_\q(V)\#_\alpha G$. 
By Theorem~\ref{thm: constant}, $\mu_1'$ gives rise to a twisted quantum
Drinfeld Hecke algebra $\mH_{\q, \kappa, \alpha}$,
where $\kappa$ is the quantum skew-symmetrization of $\mu_1'$.
By Lemma~\ref{TQDHA iff TQDHA over C[hbar]}, $\mH_{\q, \kappa, \alpha, \hbar}$ is a 
twisted quantum Drinfeld Hecke algebra over $\CC[\hbar]$.
By Theorem \ref{main-thm}, associated to $\mH_{\q, \kappa, \alpha, \hbar}$ is 
a deformation $\mu = \mu_0 + \mu_1 \hbar + \mu_2 \hbar^2 + \cdots$ of $S_\q(V)\#_\alpha G$. 
The proof of Proposition~\ref{quantum skew-symmetrization} shows that $\kappa$ is the 
quantum skew-symmetrization of $\mu_1$, and it follows from Theorem~\ref{thm: bijection} that
$\mu_1'$ is cohomologous to $\mu_1$. 
\end{proof}

\end{section}
\begin{section}{Diagonal actions} \label{diagonal}

As before, let $G$ be a finite group acting linearly on a vector space $V$
with basis $v_1, \ldots, v_n$. 
Assume that $v_1,\ldots,v_n$ are common eigenvectors for $G$.
In this case, the Hochschild cohomology $\HHD(S_\q(V), S_\q(V)\# G)$ 
was computed in \cite{NSW}. 
Let $\alpha$ be a normalized $2$-cocycle on $G$.
In this section, we use results 
from \cite{NSW} to give an explicit description of the subspace of 
$\HH^2(S_\q(V)\#_\alpha G)$ consisting of constant Hochschild $2$-cocycles.
As a consequence, we obtain a classification of 
twisted quantum Drinfeld Hecke algebras
associated to the quadruple $(G, V, \q, \alpha)$. 
 
Let $\lambda_{g,i}\in \CC$ be the scalars for which 
$\lexp{g}{v_i} = \lambda_{g,i}v_i$ for all 
$g \in G$, and all $i \in \{1,\ldots,n\}$.
For each $g \in G$, define
\begin{equation}\label{Cg}
C_g := \left\{ \gamma \in (\N \cup \{-1\})^n \mid 
\text{ for each } i \in \{1, \ldots, n\}, \;
\prod_{s=1}^n q_{is}^{\gamma_s} = \lambda_{g,i} \text{ or } \gamma_i = -1 \right\}.
\end{equation}

We recall the following from \cite{NSW}.

\begin{theorem}[\cite{NSW}]
If  $G$ acts diagonally on  $V$, then 
$\HH^{\bu}(S_\q(V),S_\q(V) \# G)$ is isomorphic to the graded vector subspace of 
$(S_\q(V) \#_G) \ot \Wedge_{\q^{-1}}(V^*)$ given by:
$$
\HH^m(S_\q(V),S_\q(V) \# G) 
\cong  \bigoplus_{g \in G}
\bigoplus_{\substack{\beta \in \O^n \\ |\beta| = m}} 
\bigoplus_{\substack{\tau \in \N^n \\ \tau - \beta \in C_g}}
\span_{\CC}\{(v^\tau t_g) \ot {(v^*)}^{\wedge \beta}\},
$$
for all $m \in \N$.
\end{theorem}

An immediate consequence is the following. 

\begin{corollary}\label{cor: abelian}
The constant Hochschild $2$-cocycles representing elements in the cohomology
$\ds \HH^2(S_\q(V), S_\q(V) \# G)$ form a vector space with basis all
\[
t_g \ot v_r^* \wedge v_s^*,
\]
where $r<s$ and $g \in G$ satisfy $q_{rr'}q_{sr'}=\lambda_{g,r'}$ for all $r' \not \in  \{r,s\}$.
\end{corollary}

Note that the $S_\q(V)$-bimodule structure of $S_\q(V) \#_\alpha G$ does not
depend on the $2$-cocycle $\alpha$, and so
$\HH^2(S_\q(V), S_\q(V) \#_\alpha G) = \HH^2(S_\q(V), S_\q(V) \# G)$.

Let $\mathcal R$ denote a complete set of representatives of conjugacy classes in $G$, 
let $C_G(a)$ denote the centralizer of $a$ in $G$, and let $[G/C_G(a)]$ denote
a complete set of representatives of left cosets of $C_G(a)$ in $G$.
In the theorem below, the notation $\delta_{i,j}$ is the Kronecker
delta. 

\begin{theorem} \label{G-invariant elements}
The constant Hochschild $2$-cocycles representing elements in the cohomology
$\ds \HH^2(S_\q(V) \#_\alpha G)$ form a vector space with basis all
\[
\sum_{g \in [G/C_G(a)]}
\frac{\alpha(g, a)}{\alpha(gag^{-1}, g)} 
\lambda_{g,r}^{-1}\lambda_{g,s}^{-1} t_{gag^{-1}} \ot v_r^* \wedge v_s^*,
\]
where $r<s$ and $a \in \mathcal{R}$ satisfy $q_{rr'}q_{sr'}=\lambda_{a,r'}$ for all $r' \not \in  \{r,s\}$,
and $\lambda_{h,r}\lambda_{h,s}=\frac{\alpha(h,a)}{\alpha(a,h)}$ for all $h \in C_G(a)$.
\end{theorem}

\begin{proof}
We will show that the space of $G$-invariant elements of the vector space given 
in Corollary~\ref{cor: abelian} is precisely the vector space stated in the theorem. 
The stated result will then follow from Theorem~\ref{thm: CGW}. 

First, we will show that the scalar 
$\frac{\alpha(g, a)}{\alpha(gag^{-1}, g)}\lambda_{g,r}^{-1}\lambda_{g,s}^{-1}$ 
is independent of choice of representative $g$ of a coset of $C_G(a)$ under
the assumption that $\lambda_{h,r}\lambda_{h,s}=\frac{\alpha(h,a)}{\alpha(a,h)}$ for all $h\in C_G(a)$.
Suppose that $gag^{-1} = g'ag'^{-1}$. Then $g'=gh$ for some $h \in C_G(a)$, and we have
\[
\frac{\alpha(g', a)}{\alpha(g'ag'^{-1}, g')}\lambda_{g',r}^{-1}\lambda_{g',s}^{-1}
= \frac{\alpha(gh, a)}{\alpha(gag^{-1}, gh)}\lambda_{g,r}^{-1}\lambda_{g,s}^{-1}\lambda_{h,r}^{-1}\lambda_{h,s}^{-1}
\]
Substituting $\lambda_{h,r}\lambda_{h,s}=\frac{\alpha(h,a)}{\alpha(a,h)}$ yields
\[
\frac{\alpha(gh, a)\alpha(a,h)}{\alpha(gag^{-1}, gh)\alpha(h,a)}\lambda_{g,r}^{-1}\lambda_{g,s}^{-1}.
\]
Applying the $2$-cocycle condition of $\alpha$ to the triple $(g,h,a)$ gives
$\frac{\alpha(gh,a)}{\alpha(h,a)} = \frac{\alpha(g,ha)}{\alpha(g,h)}$. Making
this substitution in the expression above yields
\[
\frac{\alpha(g, ha)\alpha(a,h)}{\alpha(gag^{-1}, gh)\alpha(g,h)}\lambda_{g,r}^{-1}\lambda_{g,s}^{-1}.
\]
Applying the $2$-cocycle condition of $\alpha$ to the triple $(g,a,h)$ gives
$\alpha(g,ha)\alpha(a,h)= \alpha(ga,h)\alpha(g,a)$. Making
this substitution in the expression above yields
\[
\frac{\alpha(ga, h)\alpha(g,a)}{\alpha(gag^{-1}, gh)\alpha(g,h)}\lambda_{g,r}^{-1}\lambda_{g,s}^{-1}.
\]
Finally, applying the $2$-cocycle condition of $\alpha$ to the triple $(gag^{-1},g,h)$ gives
$\frac{\alpha(ga,h)}{\alpha(gag^{-1},gh)\alpha(g,h)} = \frac{1}{\alpha(gag^{-1},g)}$. Making
this substitution in the expression above yields
\[
\frac{\alpha(g, a)}{\alpha(gag^{-1},g)}\lambda_{g,r}^{-1}\lambda_{g,s}^{-1},
\]
proving that the scalar above is independent of choice of representative $g$ of a coset of $C_G(a)$ under
the assumption that $\lambda_{h,r}\lambda_{h,s}=\frac{\alpha(h,a)}{\alpha(a,h)}$ for all $h\in C_G(a)$.
Thus, each of the alleged basis element is well-defined, and is
evidently $G$-invariant. 

Conversely, let $\eta = \sum_{a \in G} \sum \eta_{rs}^a t_a \ot v_r^* \wedge v_s^*$,
where $\eta_{rs}^a$ are scalars and the second sum runs over all $r<s$ that satisfy 
$q_{rr'}q_{sr'}=\lambda_{a,r'}$ for all $ r' \not \in \{r,s\}$.
We have
\[
\lexp{g}{\eta} = \sum_{a \in G} \eta_{rs}^a t_gt_a(t_g)^{-1} \ot 
\lexp{g}{(v_r^*)} \wedge \lexp{g}{(v_s^*)}
= \sum_{a \in G} \frac{\alpha(g, a)}{\alpha(gag^{-1}, g)} 
 \lambda_{g, r}^{-1}\lambda_{g, s}^{-1} \eta_{rs}^a t_{gag^{-1}} \ot v_r^* \wedge v_s^*.
\]
Assume that  $\eta$ is $G$-invariant. Then 
\[
\eta_{rs}^{gag^{-1}} = \frac{\alpha(g,a)}{\alpha(gag^{-1},g)} \lambda_{g, r}^{-1}\lambda_{g, s}^{-1} \eta_{rs}^a,
\]
for all $g \in G$. Letting $g=h \in C_G(a)$ yields 
\[
\lambda_{h,r}\lambda_{h,s}=\frac{\alpha(h,a)}{\alpha(a,h)},
\]
showing that $\eta$ is in the span of the alleged basis elements.
The stated result now follows from Theorem~\ref{thm: CGW}.

\end{proof}

The proof of the
following theorem involves the maps $\Theta_2^*, \R_2$, and $\Psi_2^*$
defined in Section~\ref{computing HH^2}.

\begin{theorem}
The maps $\kappa: V \times V \to \CC^\alpha G$ for which $\mH_{\q, \kappa, \alpha}$ is a 
twisted quantum Drinfeld Hecke algebra form a vector space with basis consisting of maps 
\[
f_{r,s,a}: V \times V \to \CC^\alpha G: (v_i, v_j) \mapsto 
    (\delta_{i,r} \delta_{j, s} - q_{sr} \delta_{i,s} \delta_{j, r}) \sum_{g\in [G/C_G(a)]} 
    \frac{\alpha(g, a)}{\alpha(gag^{-1}, g)} \lambda_{g,r}^{-1} \lambda_{g,s}^{-1} \ t_{gag^{-1}}
\]
where $r < s$ and $a\in {\mathcal R}$ satisfy  
$ \ q_{rr'}q_{sr'}=\lambda_{a,r'}$ for all $r' \not \in \{r,s\}$ and $\lambda_{h,r}\lambda_{h,s}
=\frac{\alpha(h,a)}{\alpha(a,h)}$ for all $h\in C_G(a)$.
\end{theorem}

\begin{proof}
Let $\eta = \sum_{g \in [G/C_G(a)]}
\frac{\alpha(g, a)}{\alpha(gag^{-1}, g)} 
\lambda_{g,r}^{-1}\lambda_{g,s}^{-1} t_{gag^{-1}} \ot v_r^* \wedge v_s^*$,
where $r<s$ and $a \in \mathcal{R}$ satisfy the conditions specified in Theorem~\ref{G-invariant elements}.
In the  proof of Theorem~\ref{thm: constant} we saw that 
$[\Theta_2^* \R_2   \Psi_2^* (\eta)](v_i \ot v_j - q_{ij} v_j \ot v_i) = \eta(v_i \wedge v_j)$,
and the latter is equal to 
\[
(\delta_{i,r} \delta_{j, s} - q_{sr} \delta_{i,s} \delta_{j, r}) \sum_{g\in [G/C_G(a)]} 
    \frac{\alpha(g, a)}{\alpha(gag^{-1}, g)} \lambda_{g,r}^{-1} \lambda_{g,s}^{-1} \ t_{gag^{-1}}.
\]
The stated result now follows from Theorem~\ref{thm: CGW} and Theorem~\ref{thm: bijection}.

\end{proof}

\end{section}
\begin{section}{Symmetric groups: Natural representations} \label{natural}

In this section, we classify twisted quantum Drinfeld Hecke algebras for 
the symmetric groups $S_n, \, n \geq 4$, acting naturally on a vector space of dimension $n$.

Consider the natural action of $S_n$ on a 
vector space $V$ with ordered basis $v_1,\ldots,v_n$. 
Let $\q:= (q_{ij})_{1 \leq i,j \leq n}$ denote a tuple of
nonzero scalars for which $q_{ii}=1$ and $q_{ji}=q_{ij}^{-1}$ for all $i,j$.
The action of $S_n$
extends to an action on the quantum symmetric algebra $S_\q(V)$ by automorphisms
if and only if either $q_{ij}=1$ for all $i,j$, or 
$q_{ij}=-1$ for all $i \neq j$. The tuple corresponding to
the former will be denoted by $\mathbf{1}$, and
the tuple corresponding to the latter will be denoted by $\mathbf{-1}$. 
The action of $S_n$ on $V$ extends to an action on
the quantum exterior algebra $\Wedge_{\mathbf{-1}}$ by automorphisms.
Note that the algebra $\Wedge_{\mathbf{-1}}$ is commutative.

The Schur multiplier $\H^2(S_n,\CC^\times)$ of the symmetric group $S_n$ is
isomorphic to $\Z/2\Z$ for all $n \geq 4$ \cite{S}.  Let $\alpha$ be a $2$-cocycle
on $S_n$, and let $[\alpha]$ denote the image of $\alpha$ in $\H^2(S_n,\CC^\times)$.
A classification of twisted quantum Drinfeld Hecke algebras for 
$S_n$, acting naturally on a vector space of dimension $n$,
is given in  \cite{RS} for $[\alpha]=1$ and $\q=\mathbf{1}$,
in \cite{W} for $[\alpha] \neq 1$ and $\q=\mathbf{1}$, and
in \cite{NW} for $[\alpha]=1$ and $\q=\mathbf{-1}$.
The goal of this section is to address the remaining case: $[\alpha] \neq 1$ and $\q=\mathbf{-1}$.

Next, we recall a Schur covering group of $S_n$.
We will use it to obtain a cohomologically
nontrivial $2$-cocycle on $S_n$. 
Let $T_n$ be the group with generators
$t_1, \ldots, t_{n-1},z$ and relations
\[
\begin{aligned}[2]
z^2 &= 1; &&\\
t_r^2 &= 1, && 1 \leq r \leq n-1;\\
t_rt_s &= t_st_rz, && \text{for } |r-s|>1 \text{ and } 1 \leq r,s \leq n-1;\\
t_rt_{r+1}t_r &= t_{r+1}t_rt_{r+1}, && 1 \leq r \leq n-2;\\
zt_r &= t_rz, && 1 \leq r \leq n-1.\\
\end{aligned}
\]
The group $T_n$ is a central extension of $S_n$ by $\la z\ra$:
\[
1 \to \la z \ra \to T_n \xrightarrow{p} S_n \to 1,
\]
where the surjection $p$ sends $z$ to $1$ and sends $t_r$ to the transposition $(r \, r+1)$.
The group $T_n$ is a Schur covering group of $S_n$ \cite{S}. 

We define certain distinguished elements of $T_n$: For every
$r,s \in \{1, \ldots, n\}$, $r \neq s$, denote by $[rs]$ the
element of $T_n$ defined recursively as follows:
\[
\begin{aligned}[2]
[r \, r+1] &:= t_r, &&\\
[rs] &:= t_r[r+1 \, s]t_rz &&\text{if } r < s-1,\\
[rs] &:= [sr]z &&\text{if } r>s.
\end{aligned}
\]
Note that $p([rs]) = (rs)$.

Next, we define a section $u: S_n \to T_n: \sigma \mapsto u_\sigma$ of the surjection $p: T_n \to S_n$.
If $\sigma \in S_n$ is the $k$-cycle $(a_1, \ldots, a_k)$, where 
$a_1, \ldots, a_k \in \{1, \ldots, n\}$ and $a_1$ is the smallest
element of the set $\{a_1, \ldots, a_k\}$, then define
\[
u_\sigma := [a_1a_k][a_1a_{k-1}] \cdots [a_1a_2].
\]
If $\sigma \in S_n$ is the product $(a_1, \ldots, a_k)(b_1, \ldots, b_l)\cdots$ of disjoint cycles,
where $a_1$ is the smallest element of the set $\{a_1, \ldots, a_k\}$, $b_1$ is the
smallest element of the set $\{b_1, \ldots, b_l\}$, and so on, and
$a_1 < b_1 < \cdots$, then define
\[
u_\sigma := u_{(a_1, \ldots, a_k)}u_{(b_1, \ldots, b_l)} \cdots. 
\]
It is evident that $u:S_n \to T_n$ is a section, that is, $pu=\id_{S_n}$.

Consider any irreducible representation of the group $T_n$. Since the element $z$ is central
and has order two, it must necessarily act on this representation as multiplication by either $1$ or $-1$.
Assume the latter. In this case, we obtain a cohomologically nontrivial
(normalized) $2$-cocyle $\alpha: S_n \times S_n \to \CC^\times$ defined by
\begin{equation}\label{defn of alpha}
\alpha(\sigma,\tau):=
\begin{cases}
1 & \text{ if } u_\sigma u_\tau u_{\sigma \tau}^{-1} = 1,\\
-1 & \text{ if } u_\sigma u_\tau u_{\sigma \tau}^{-1} = z,
\end{cases}
\end{equation}
for all $\sigma, \tau \in S_n$.

Our goal is to classify twisted quantum Drinfeld Hecke algebras
associated to the quadruple $(S_n, V, \mathbf{-1}, \alpha)$, where $V$ is 
the natural representation of $S_n$ and $\mathbf{-1}$ is the tuple 
defined earlier in this section. To this end, in what follows, we establish
several lemmas that will aid in accomplishing our goal.

Since the subgroup $\la z \ra$ of $T_n$ is central, there is an action
of $S_n$ on $T_n$ induced by conjugation. If $\sigma$ belongs to $S_n$ and $\nu$ belongs to $T_n$,
we denote by $\sigma \rt \nu$ the result of $\sigma$ acting upon $\nu$.
We have $\sigma \rt \nu = \hat{\sigma} \nu (\hat{\sigma})^{-1}$,
where $\hat{\sigma}$ is any element in the set $p^{-1}(\sigma)$.

For each $\sigma \in S_n$, let $\epsilon(\sigma)$ denote the signature of $\sigma$:
\[
\epsilon(\sigma) =
\begin{cases}
0 & \text{ if $\sigma$ is an even permutation}, \\
1 & \text{ if $\sigma$ is an odd permutation}.
\end{cases}
\]
The following result from \cite{V} will be put to use shortly.

\begin{lemma} \label{V lemma}
For all distinct $r,s \in \{1, \ldots, n\}$, and all $\sigma \in S_n$, we have
\[
\sigma \rt [rs]= [\sigma(r)\sigma(s)]z^{\epsilon(\sigma)}.
\]
\end{lemma}

\noindent For later use, we record the following two lemmas.

\begin{lemma}\label{3-cycle relation}
For all distinct $r,r',s,s' \in \{1, \ldots, n\}$, we have
\[
[rs][sr']z = [rr'][rs] = [r's][rr']z.
\]
\end{lemma}

\begin{proof}
We have $[rs]^{-1}[rr'][rs] = (rs)^{-1} \rt [rr'] = (rs) \rt [rr']$, and
by Lemma~\ref{V lemma} the last expression equals $[sr']z$, proving the first equality.
The second equality is proved similarly.
\end{proof}

\begin{lemma}\label{commuting relation}
For all distinct $r,r',s,s' \in \{1, \ldots, n\}$, we have
\[
[rs][r's'] = [r's'][rs]z.
\]
\end{lemma}

\begin{proof}
We have $[rs][r's'][rs]^{-1} = (rs) \rt [r's']$, and
by Lemma~\ref{V lemma} the last expression equals $[r's']z$.
\end{proof}

For all distinct $r,s,r',s' \in \{1, \ldots, n\}$, let
$d(r,s,r',s')$ denote the number of inequalities 
\[
\begin{aligned}[1]
\min\{r,s\} &> \min\{r',s'\},\\
r &> s,\\
r' &> s'.
\end{aligned}
\]
that hold.
For all distinct $r,s,r',s' \in \{1, \ldots, n\}$, and all $\sigma \in S_n$, 
define
\[
d_\sigma(r,s,r',s') := d(\sigma(r),\sigma(s),\sigma(r'),\sigma(s')).
\]

\noindent For later use, we record the following obvious result.

\begin{lemma}\label{props of d}
For all distinct $r,s,r',s' \in \{1, \ldots, n\}$, we have
\[
|d(r,s,r',s') - d(r,s,s',r')| = 1 =
|d(r,s,r',s') - d(s,r,r',s')|.
\]
\end{lemma}

\noindent We will need the following lemma. It is a generalization of \cite[Lemma~3.7]{V}.

\begin{lemma}\label{section properties}
Let $\sigma$ be any element of $S_n$.
\begin{enumerate}
\item[(a)] For all $r,s \in \{1, \ldots, n\}$ with $r<s$, we have
\[
\sigma \rt u_{(rs)} =
\begin{cases}
u_{\sigma (rs) \sigma^{-1}} z^{\epsilon(\sigma)} & \text{ if } \sigma(r) < \sigma(s),\\
u_{\sigma (rs) \sigma^{-1}} z^{\epsilon(\sigma)+1} & \text{ if } \sigma(r) > \sigma(s).
\end{cases}
\]
\item[(b)] For all distinct $r,s,r',s' \in \{1, \ldots, n\}$ with $r<s$, $r'<s'$, and 
$r<r'$, we have
\[
\sigma \rt u_{(rs)(r's')} = u_{\sigma (rs)(r's') \sigma^{-1}} z^{d_\sigma(r,s,r',s')}
\]
\item[(c)] For all distinct $r,s,r' \in \{1, \ldots, n\}$ with $r<s$ and $r<r'$, we have
\[
\sigma \rt u_{(rsr')} = u_{\sigma (rsr') \sigma^{-1}}.
\]
\end{enumerate}
\end{lemma}

\begin{proof}
(a) By Lemma~\ref{V lemma}, $\sigma \rt u_{(rs)} = \sigma \rt [rs] = [\sigma(r)\sigma(s)]z^{\epsilon(\sigma)}$.
If $\sigma(r) < \sigma(s)$, then 
\[
[\sigma(r)\sigma(s)]z^{\epsilon(\sigma)} = u_{(\sigma(r)\sigma(s))}z^{\epsilon(\sigma)}
= u_{\sigma(rs)\sigma^{-1}}z^{\epsilon(\sigma)}.
\]
If $\sigma(r) > \sigma(s)$, then 
\[
[\sigma(r)\sigma(s)]z^{\epsilon(\sigma)} = [\sigma(s)\sigma(r)]z^{\epsilon(\sigma)+1}
= u_{(\sigma(s)\sigma(r))}z^{\epsilon(\sigma)+1}
= u_{\sigma(rs)\sigma^{-1}}z^{\epsilon(\sigma)+1}.
\]

\noindent (b) Again, by Lemma~\ref{V lemma}, 
\[
\begin{aligned}[1]
\sigma \rt u_{(rs)(r's')} = \sigma \rt [rs][r's'] &= (\sigma \rt [rs])(\sigma \rt [r's']) \\
&= [\sigma(r)\sigma(s)]z^{\epsilon(\sigma)}[\sigma(r')\sigma(s')]z^{\epsilon(\sigma)}\\
&= [\sigma(r)\sigma(s)][\sigma(r')\sigma(s')].
\end{aligned}
\]
If $\min\{\sigma(r),\sigma(s)\} > \min\{\sigma(r'),\sigma(s')\}$, then using 
Lemma~\ref{commuting relation} we rewrite the product above as $[\sigma(r')\sigma(s')][\sigma(r)\sigma(s)]z$.
If $\sigma(r) > \sigma(s)$, then we replace $[\sigma(r)\sigma(s)]$ by $[\sigma(s)\sigma(r)]z$.
Similarly, if $\sigma(r') > \sigma(s')$, then we replace $[\sigma(r')\sigma(s')]$ by $[\sigma(s')\sigma(r')]z$.
Since the element $z$ has order two, the stated result follows. For example, suppose that
$d_\sigma(r,s,r',s')=3$. Then $\sigma(r)>\sigma(s), \sigma(r')>\sigma(s')$ and $\sigma(s)>\sigma(s')$,
and in this case we write
\[
\begin{aligned}[1]
[\sigma(r)\sigma(s)][\sigma(r')\sigma(s')] = [\sigma(r')\sigma(s')][\sigma(r)\sigma(s)]z
&= [\sigma(s')\sigma(r')]z[\sigma(s)\sigma(r)]zz \\
&= [\sigma(s')\sigma(r')][\sigma(s)\sigma(r)]z\\
&= u_{(\sigma(s')\sigma(r'))(\sigma(s)\sigma(r))}z \\
&= u_{\sigma(rs)(r's')\sigma^{-1}}z.
\end{aligned}
\]

\noindent (c) Again, by Lemma~\ref{V lemma}, 
\[
\begin{aligned}[1]
\sigma \rt u_{(rsr')} = \sigma \rt [rr'][rs] &= (\sigma \rt [rr'])(\sigma \rt [rs]) \\
&= [\sigma(r)\sigma(r')]z^{\epsilon(\sigma)}[\sigma(r)\sigma(s)]z^{\epsilon(\sigma)}\\
&= [\sigma(r)\sigma(r')][\sigma(r)\sigma(s)].
\end{aligned}
\]
\noindent Case ($\text{c}_1$): $\sigma(r)<\sigma(r')$ and $\sigma(r)<\sigma(s)$. In this case,
\[
[\sigma(r)\sigma(r')][\sigma(r)\sigma(s)] = u_{(\sigma(r)\sigma(s)\sigma(r'))} = u_{\sigma(rsr')\sigma^{-1}}.
\]

\noindent Case ($\text{c}_2$): Either $\sigma(s)<\sigma(r)<\sigma(r')$ or $\sigma(s)<\sigma(r')<\sigma(r)$. 
Using the first equality of Lemma~\ref{3-cycle relation},
\[
\begin{aligned}[1]
[\sigma(r)\sigma(r')][\sigma(r)\sigma(s)] = [\sigma(r)\sigma(s)][\sigma(s)\sigma(r')]z
&= [\sigma(s)\sigma(r)]z[\sigma(s)\sigma(r')]z \\
&= u_{(\sigma(s)\sigma(r')\sigma(r))}\\
&= u_{\sigma(rsr')\sigma^{-1}}.
\end{aligned}
\]

\noindent Case ($\text{c}_3$): Either $\sigma(r')<\sigma(r)<\sigma(s)$ or $\sigma(r')<\sigma(s)<\sigma(r)$. 
Using the second equality of Lemma~\ref{3-cycle relation},
\[
\begin{aligned}[1]
[\sigma(r)\sigma(r')][\sigma(r)\sigma(s)] = [\sigma(r')\sigma(s)][\sigma(r)\sigma(r')]z
&= [\sigma(r')\sigma(s)][\sigma(r')\sigma(r)]zz \\
&= u_{(\sigma(r')\sigma(r)\sigma(s))}\\
&= u_{\sigma(rsr')\sigma^{-1}}.
\end{aligned}
\]

\end{proof}

We now turn our attention to the Hochschild cohomology of $S_\mathbf{-1}(V)\#_\alpha S_n$.
We begin with the following, which is Theorem~6.8 from \cite{NW}.

\begin{theorem}[\cite{NW}]
\label{thm: constant cocycles m=1}
Assume that $n \geq 4$. The constant Hochschild 2-cocycles 
representing elements in $\HH^2(S_\mathbf{-1}(V), S_\mathbf{-1}(V)\# S_n)$
form a vector subspace of $(S_\mathbf{-1}(V) \# G) \ot \Wedge_{\mathbf{-1}}(V^*)$ with basis all
\[
\begin{aligned}[2]
 \eta_1 &= t_1\ot v_r^*\wedge v_s^* &&  (r<s),\\
 \eta_2 &= t_{(rs)}\ot v_r^*\wedge v_s^* && (r<s), \\
 \eta_3 &= t_{(rs)}\ot (v_r^*\wedge v_{r'}^* + v_s^*\wedge v_{r'}^*) && (r<s),\\
 \eta_4 &= t_{(rs)(r's')} \ot (v_r^*\wedge v_{r'}^* + v_r^*\wedge v_{s'}^* + v_s^*\wedge v_{r'}^* 
               + v_s^*\wedge v_{s'}^*) && (r<s, r'<s', r<r'),\\
 \eta_5 &= t_{(rs{r'})}\ot (v_r^*\wedge v_s^* + v_s^*\wedge v_{r'}^* + v_r^*\wedge v_{r'}^*) && (r<s, r<r').
\end{aligned}
\]
\end{theorem}

Note that the $S_\mathbf{-1}(V)$-bimodule structure of $S_\mathbf{-1}(V) \#_\alpha G$ does not
depend on the $2$-cocycle $\alpha$, and so
$\HH^2(S_\mathbf{-1}(V), S_\mathbf{-1}(V) \#_\alpha G) = \HH^2(S_\mathbf{-1}(V), S_\mathbf{-1}(V) \# G)$.

The lemma below involves the maps $\Theta_2^*, \R_2$, and $\Psi_2^*$
defined in Section~\ref{computing HH^2}. Recall that the image of 
an element $\sigma \in S_n$ in the twisted group algebra $\CC^\alpha S_n$ 
is denoted by $t_\sigma$. Also, recall the definition of 
the $2$-cocycle $\alpha$ given in \eqref{defn of alpha}.

\begin{lemma}\label{image of vi ot vj}
We have
\[
 \left[(\Theta_2^* \R_2  \Psi_2^*)(\eta_a)\right](v_i \ot v_j)=
\begin{cases}
\frac{1}{n(n-1)} t_1 & \text{ if } a=1,  \\
0 & \text{ if $a=2$},  \\
0 & \text{ if $a=3$ and $n \geq 5$},  \\
0 & \text{ if $a=4$},  \\
\frac{1}{n(n-1)(n-2)}\sum_{k\neq i,j} (2t_{(ijk)} + t_{(ikj)}) & \text{ if } a=5, 
\end{cases}
\]
for all $i \neq j$.
\end{lemma}

\begin{proof}
Using \eqref{composition}, 
\[
\begin{aligned}[1]
\left[(\Theta_2^* \R_2  \Psi_2^*)(\eta_1)\right](v_i \ot v_j) 
&=\frac{1}{n!} \sum_{\sigma \in S_n} 
\lexp{\sigma}{\left(\eta_1(\Psi_2(1 \ot v_{\sigma^{-1}(i)}\ot v_{\sigma^{-1}(j)} \ot 1))\right)}\\
&=\frac{1}{n!} \sum_{\substack{\sigma \in S_n \\ \sigma^{-1}(i) < \sigma^{-1}(j)}} 
\lexp{\sigma}{\left(\eta_1(1 \ot 1 \ot v_{\sigma^{-1}(i)}\ot v_{\sigma^{-1}(j)})\right)}\\
&=\frac{1}{n!} \sum_{\substack{\sigma \in S_n \\ \sigma(r)=i,\sigma(s)=j}} 
\lexp{\sigma}{\left(t_1\right)}\\
&=\frac{1}{n(n-1)} t_1.
\end{aligned}
\]

Similarly, 
\[
\left[(\Theta_2^* \R_2  \Psi_2^*)(\eta_2)\right](v_i \ot v_j) 
=\frac{1}{n!} \sum_{\substack{\sigma \in S_n \\ \sigma(r)=i,\sigma(s)=j}} 
\lexp{\sigma}{\left(t_{(rs)}\right)}.
\]
Applying the conjugation action in $\CC^\alpha G$, we get
\[
\frac{1}{n!} \sum_{\substack{\sigma \in S_n \\ \sigma(r)=i,\sigma(s)=j}} 
\frac{\alpha(\sigma,(rs))}{\alpha(\sigma(rs)\sigma^{-1},\sigma)}t_{\sigma(rs)\sigma^{-1}}
=\left(\frac{1}{n!} \sum_{\substack{\sigma \in S_n \\ \sigma(r)=i,\sigma(s)=j}} 
\frac{\alpha(\sigma,(rs))}{\alpha((ij),\sigma)}\right)t_{(ij)}.
\]
The scalar $\frac{\alpha(\sigma,(rs))}{\alpha((ij),\sigma)}$ in the summation above is determined
by the following element of $T_n$:
\[
u_\sigma u_{(rs)} u_{\sigma(rs)}^{-1} u_{(ij)\sigma} u_\sigma^{-1} u_{(ij)}^{-1}
= u_\sigma u_{(rs)} u_\sigma^{-1} u_{\sigma(rs)\sigma^{-1}}^{-1}.
\]
By part (a) of Lemma~\ref{section properties},
\[
u_\sigma u_{(rs)} u_\sigma^{-1} u_{\sigma(rs)\sigma^{-1}}^{-1}=
\begin{cases}
z^{\epsilon(\sigma)} & \text{ if } i<j,\\
z^{\epsilon(\sigma)+1} & \text{ if } i>j.
\end{cases}
\]
Since $n$ is assumed to be greater than or equal to  $4$, 
the set $\{\sigma \in S_n \mid \sigma(r)=i, \sigma(s)=j\}$ contains
an equal number of odd and even permutations, and so  
\[
\sum_{\substack{\sigma \in S_n \\ \sigma(r)=i,\sigma(s)=j}} 
\frac{\alpha(\sigma,(rs))}{\alpha((ij),\sigma)}=0,
\]
proving that $\left[(\Theta_2^* \R_2  \Psi_2^*)(\eta_2)\right](v_i \ot v_j)=0$.

Next, we consider the $a=3$ case. In addition to the stated assumption $r<s$, 
assume further that $r<r'$ and $s<r'$. The other cases can be handled similarly.
We have
\[
\left[(\Theta_2^* \R_2  \Psi_2^*)(\eta_3)\right](v_i \ot v_j)
= \frac{1}{n!} \sum_{\substack{\sigma \in S_n \\ \sigma(r)=i,\sigma(r')=j}} 
\lexp{\sigma}{\left(t_{(rs)}\right)}
+ \frac{1}{n!} \sum_{\substack{\sigma \in S_n \\ \sigma(s)=i,\sigma(r')=j}} 
\lexp{\sigma}{\left(t_{(rs)}\right)}.
\]
Applying the conjugation action in $\CC^\alpha G$, we get
\[
\begin{aligned}[1]
&\frac{1}{n!} \sum_{\substack{\sigma \in S_n \\ \sigma(r)=i,\sigma(r')=j}} 
\frac{\alpha(\sigma, (rs))}{\alpha((i\sigma(s)), \sigma)} t_{(i\sigma(s))}
+ \frac{1}{n!} \sum_{\substack{\sigma \in S_n \\ \sigma(s)=i,\sigma(r')=j}} 
\frac{\alpha(\sigma, (rs))}{\alpha((\sigma(r)i), \sigma)} t_{(\sigma(r)i)}\\
&=\frac{1}{n!} \sum_{k \neq i,j} \left(
\sum_{\substack{\sigma \in S_n \\ \sigma(r)=i,\sigma(r')=j, \sigma(s)=k}} 
\frac{\alpha(\sigma, (rs))}{\alpha((ik), \sigma)}
+ \sum_{\substack{\sigma \in S_n \\ \sigma(s)=i,\sigma(r')=j, \sigma(r)=k}} 
\frac{\alpha(\sigma, (rs))}{\alpha((ik), \sigma)}
\right) t_{(ik)}.
\end{aligned}
\]
The scalar $\frac{\alpha(\sigma,(rs))}{\alpha((ik),\sigma)}$ in the first of
the two inner summations above is determined
by the element $u_\sigma u_{(rs)} u_\sigma^{-1} u_{\sigma(rs)\sigma^{-1}}^{-1}$ of $T_n$.
Again, by part (a) of Lemma~\ref{section properties},
\[
u_\sigma u_{(rs)} u_\sigma^{-1} u_{\sigma(rs)\sigma^{-1}}^{-1}=
\begin{cases}
z^{\epsilon(\sigma)} & \text{ if } i<k,\\
z^{\epsilon(\sigma)+1} & \text{ if } i>k.
\end{cases}
\]
Since $n$ is assumed to be greater than or equal to  $5$, 
the set $\{\sigma \in S_n \mid \sigma(r)=i, \sigma(r')=j, \sigma(s)=k\}$ contains
an equal number of odd and even permutations, and so  
\[
\sum_{\substack{\sigma \in S_n \\ \sigma(r)=i,\sigma(r')=j,\sigma(s)=k}} 
\frac{\alpha(\sigma,(rs))}{\alpha((ik),\sigma)}=0.
\]
Similarly, 
\[
\sum_{\substack{\sigma \in S_n \\ \sigma(s)=i,\sigma(r')=j,\sigma(r)=k}} 
\frac{\alpha(\sigma,(rs))}{\alpha((ik),\sigma)}=0,
\]
and it follows that $\left[(\Theta_2^* \R_2  \Psi_2^*)(\eta_3)\right](v_i \ot v_j)=0$.

For the $a=4$ case, in addition to the stated assumptions $r<s, r'<s',r<r'$, 
assume further that $r<s', s<r',$ and $s<s'$. The other cases can be handled similarly.
We have
\[
\begin{aligned}[1]
\left[(\Theta_2^* \R_2  \Psi_2^*)(\eta_4)\right](v_i \ot v_j) 
&=\frac{1}{n!} \sum_{\substack{\sigma \in S_n \\ \sigma(r)=i,\sigma(r')=j}} 
\lexp{\sigma}{\left(t_{(rs)(r's')}\right)}
+ \frac{1}{n!} \sum_{\substack{\sigma \in S_n \\ \sigma(r)=i,\sigma(s')=j}} 
\lexp{\sigma}{\left(t_{(rs)(r's')}\right)} \\
& \hspace{0.725in} + \; \frac{1}{n!} \sum_{\substack{\sigma \in S_n \\ \sigma(s)=i,\sigma(r')=j}} 
\lexp{\sigma}{\left(t_{(rs)(r's')}\right)}
+ \frac{1}{n!} \sum_{\substack{\sigma \in S_n \\ \sigma(s)=i,\sigma(s')=j}} 
\lexp{\sigma}{\left(t_{(rs)(r's')}\right)}
\end{aligned}
\]
Applying the conjugation action in $\CC^\alpha G$, we get
\[
\begin{aligned}
&\frac{1}{n!} \Bigg( \sum_{\substack{\sigma \in S_n \\ \sigma(r)=i,\sigma(r')=j}} 
\frac{\alpha(\sigma, (rs)(r's'))}{\alpha((i\sigma(s))(j\sigma(s')), \sigma)} t_{(i\sigma(s))(j\sigma(s'))}
+ \sum_{\substack{\sigma \in S_n \\ \sigma(r)=i,\sigma(s')=j}}
\frac{\alpha(\sigma, (rs)(r's'))}{\alpha((i\sigma(s))(\sigma(r')j), \sigma)} t_{(i\sigma(s))(\sigma(r')j)} \\
&+ \sum_{\substack{\sigma \in S_n \\ \sigma(s)=i,\sigma(r')=j}} 
\frac{\alpha(\sigma, (rs)(r's'))}{\alpha((\sigma(r)i)(j\sigma(s')), \sigma)} t_{(\sigma(r)i)(j\sigma(s'))} 
+ \sum_{\substack{\sigma \in S_n \\ \sigma(s)=i,\sigma(s')=j}} 
\frac{\alpha(\sigma, (rs)(r's'))}{\alpha((\sigma(r)i)(\sigma(r')j), \sigma)} t_{(\sigma(r)i)(\sigma(r')j)} \Bigg)\\
&=\frac{1}{n!} \sum_{k,l \not \in \{i,j\}} \Bigg(
\sum_{\substack{\sigma \in S_n \\ \sigma(r)=i,\sigma(r')=j \\ \sigma(s)=k,\sigma(s')=l}} 
\frac{\alpha(\sigma, (rs)(r's'))}{\alpha((ik)(jl), \sigma)}
+ \sum_{\substack{\sigma \in S_n \\ \sigma(r)=i,\sigma(s')=j \\ \sigma(s)=k,\sigma(r')=l}} 
\frac{\alpha(\sigma, (rs)(r's'))}{\alpha((ik)(jl), \sigma), \sigma)} \\
& \hspace{1.35in} + \; \sum_{\substack{\sigma \in S_n \\ \sigma(s)=i,\sigma(r')=j \\ \sigma(r)=k,\sigma(s')=l}} 
\frac{\alpha(\sigma, (rs)(r's'))}{\alpha((ik)(jl), \sigma), \sigma)}
+ \sum_{\substack{\sigma \in S_n \\ \sigma(s)=i,\sigma(s')=j \\ \sigma(r)=k,\sigma(r')=l}} 
\frac{\alpha(\sigma, (rs)(r's'))}{\alpha((ik)(jl), \sigma), \sigma)}
\Bigg) t_{(ik)(jl)}.
\end{aligned}
\]
The scalar $\frac{\alpha(\sigma,(rs)(r's'))}{\alpha((ik)(jl),\sigma)}$ in the first of
the four inner summations above is determined
by the element $u_\sigma u_{(rs)(r's')} u_\sigma^{-1} u_{\sigma(rs)(r's')\sigma^{-1}}^{-1}$ of $T_n$.
By part (b) of Lemma~\ref{section properties},
\[
u_\sigma u_{(rs)(r's)} u_\sigma^{-1} u_{\sigma(rs)(r's')\sigma^{-1}}^{-1}=
z^{d_\sigma(r,s,r's')} = z^{d(i,k,j,l)}.
\]
Thus, 
\[
\sum_{\substack{\sigma \in S_n \\ \sigma(r)=i,\sigma(r')=j \\ \sigma(s)=k,\sigma(s')=l}} 
\frac{\alpha(\sigma, (rs)(r's'))}{\alpha((ik)(jl), \sigma)}
= (n-4)! (-1)^{d(i,k,j,l)}.
\]
Similarly, the second, third, and fourth summations are equal to $(n-1)!$ times
$(-1)^{d(i,k,l,j)}$, $(-1)^{d(k,i,j,l)}$, $(-1)^{d(k,i,l,j)}$, respectively.
It follows, from Lemma~\ref{props of d} that the sum of the four summations above is 
equal to zero, and so
$\left[(\Theta_2^* \R_2  \Psi_2^*)(\eta_4)\right](v_i \ot v_j)=0$.

Finally, for the $a=5$ case, in addition to the stated assumptions $r<s, r<r'$, 
assume further that $s<r'$. Again, the other case can be handled similarly.
We have\\\\
$\left[(\Theta_2^* \R_2  \Psi_2^*)(\eta_5)\right](v_i \ot v_j)$
\[
\begin{aligned}[1]
&=\frac{1}{n!} \sum_{\substack{\sigma \in S_n \\ \sigma(r)=i,\sigma(s)=j}} 
\lexp{\sigma}{\left(t_{(rsr')}\right)}
+ \frac{1}{n!} \sum_{\substack{\sigma \in S_n \\ \sigma(s)=i,\sigma(r')=j}} 
\lexp{\sigma}{\left(t_{(rsr')}\right)}
+ \frac{1}{n!} \sum_{\substack{\sigma \in S_n \\ \sigma(r)=i,\sigma(r')=j}} 
\lexp{\sigma}{\left(t_{(rsr')}\right)}\\
&=\frac{1}{n!} \sum_{\substack{\sigma \in S_n \\ \sigma(r)=i,\sigma(s)=j}} 
\frac{\alpha(\sigma, (rsr'))}{\alpha((ij\sigma(r')), \sigma)} t_{(ij\sigma(r'))}
+ \frac{1}{n!} \sum_{\substack{\sigma \in S_n \\ \sigma(s)=i,\sigma(r')=j}} 
\frac{\alpha(\sigma, (rsr'))}{\alpha((\sigma(r)ij)), \sigma)} t_{(\sigma(r)ij)} \\
& \hspace{0.725in} + \; \frac{1}{n!} \sum_{\substack{\sigma \in S_n \\ \sigma(r)=i,\sigma(r')=j}} 
\frac{\alpha(\sigma, (rsr'))}{\alpha((i\sigma(s)j), \sigma)} t_{(i\sigma(s)j)}\\
&=\frac{1}{n!} \sum_{k,l \not \in \{i,j\}} \Bigg[\Bigg(
\sum_{\substack{\sigma \in S_n \\ \sigma(r)=i,\sigma(s)=j,\sigma(r')=k}} 
\frac{\alpha(\sigma, (rsr'))}{\alpha((ijk), \sigma)}
+ \sum_{\substack{\sigma \in S_n \\ \sigma(s)=i,\sigma(r')=j,\sigma(r)=k}} 
\frac{\alpha(\sigma, (rsr'))}{\alpha((ijk)), \sigma)} \Bigg) t_{(ijk)}\\
& \hspace{3.2in} + \; \sum_{\substack{\sigma \in S_n \\ \sigma(r)=i,\sigma(r')=j,\sigma(s)=k}} 
\frac{\alpha(\sigma, (rsr'))}{\alpha((ikj), \sigma)}t_{(ikj)} \Bigg].
\end{aligned}
\]
The scalar $\frac{\alpha(\sigma,(rsr'))}{\alpha((ijk),\sigma)}$ in the first of
the three inner summations above is determined
by the element $u_\sigma u_{(rsr')} u_\sigma^{-1} u_{\sigma(rsr')\sigma^{-1}}^{-1}$ of $T_n$.
By part (c) of Lemma~\ref{section properties},
\[
u_\sigma u_{(rsr')} u_\sigma^{-1} u_{\sigma(rsr')\sigma^{-1}}^{-1} = 1.
\]
Thus, 
\[
\sum_{\substack{\sigma \in S_n \\ \sigma(r)=i,\sigma(s)=j \sigma(r')=k}} 
\frac{\alpha(\sigma, (rsr')}{\alpha((ijk), \sigma)}
= (n-3)!.
\]
Similarly, the second and third summations are also equal to $(n-3)!$.
It follows that
$\left[(\Theta_2^* \R_2  \Psi_2^*)(\eta_5)\right](v_i \ot v_j)= 
\frac{1}{n(n-1)(n-2)}\sum_{k\neq i,j} (2t_{(ijk)} + t_{(ikj)})$.

\end{proof}

\noindent Combining Theorems \ref{thm: constant cocycles m=1}, \ref{thm: CGW}, 
\ref{thm: bijection}, and Lemma~\ref{image of vi ot vj}
establishes the following.

\begin{theorem}
Assume that $n \geq 5$.
The maps $\kappa:V \times V \to \CC^\alpha S_n$ for which $\mH_{\mathbf{-1}, \kappa, \alpha}$ is a twisted 
quantum Drinfeld Hecke algebra form
a two-dimensional vector space with basis consisting of bilinear maps
$\kappa_1: V \times V \to \CC^\alpha S_n$ and $\kappa_2: V \times V \to \CC^\alpha S_n$
determined by
\[
\begin{aligned}[1]
\kappa_1(v_i,v_j) &=  t_1,\\
\kappa_2(v_i,v_j) &= \sum_{k\neq i,j} (t_{(ijk)} + t_{(ikj)}),
\end{aligned}
\]
for all $i \neq j$.
\end{theorem}

\end{section}


\end{document}